\documentclass[reqno]{amsart}
\usepackage{amsfonts}
\usepackage{amsthm}
\usepackage{amsmath}
\usepackage{amscd}

\usepackage{amsfonts}
\usepackage{euscript}
\usepackage{polynom}
\usepackage{amssymb}
\usepackage{comment}
\usepackage[normalem]{ulem} 

\usepackage{color}

\usepackage[latin2]{inputenc}
\usepackage{t1enc}
\usepackage{hyperref}
\usepackage{indentfirst}
\usepackage{graphicx}
\graphicspath{ {./figures/} }
\usepackage{pict2e}
\usepackage{epic}
\numberwithin{equation}{section}
\usepackage[margin=2.9cm]{geometry}
\usepackage{epstopdf} 
\usepackage{caption}
\usepackage{subcaption}
\theoremstyle{plain}
\newtheorem{theorem}{Theorem}[section]
\newtheorem{lemma}[theorem]{Lemma}

\newtheorem{prop}[theorem]{Proposition}

 \theoremstyle{definition}
\newtheorem{Def}[theorem]{Definition}

\newtheorem{Rem}[theorem]{Remark}
\newtheorem{?}[theorem]{Problem}

\newtheorem{assum}{Assumption}

\newcommand{\pa}{\partial}

\newenvironment{proof1}[1]{\begin{trivlist} \item[] {\em Proof of #1:}}{\newline \textcolor{white}{.}\hfill $\Box$
                      \end{trivlist}}

\begin{document}

\title{Nodal Set Openings on Perturbed Rectangular Domains}


\author[T. Beck]{Thomas Beck}
\email{tbeck7@fordham.edu}
\author[M. Gupta]{Marichi Gupta}
\email{marichi.gupta@gmail.com}
\author[J.L. Marzuola]{Jeremy Marzuola}
\email{marzuola@email.unc.edu}

\begin{abstract}
We study the effects of perturbing the boundary of a rectangle on the nodal sets of eigenfunctions of the Laplacian. Namely, for a rectangle of a given aspect ratio $N$, we identify the first Dirichlet mode to feature a crossing in its nodal set and perturb one of the sides of the rectangle by a close to flat, smooth curve. Such perturbations will often ``open'' the crossing in the nodal set, splitting it into two curves, and we study the separation between these curves and their regularity. The main technique used is an approximate separation of variables that allows us to restrict study to the first two Fourier modes in an eigenfunction expansion. We show how the nature of the boundary perturbation provides conditions on the orientation of the opening and estimates on its size. In particular, several features of the perturbed nodal set are asymptotically independent of the aspect ratio, which contrasts with prior works. Numerical results supporting our findings are also presented.
\end{abstract}

\maketitle

\section{Introduction}
Modes of vibration in various media is a fundamental physical phenomenon and mathematicians have developed sophisticated theories and tools to better understand it. Of particular interest are the points that stay fixed when the object vibrates at a fundamental frequency. These correspond to the zero sets of Laplace eigenfunctions, called \textit{nodal sets}. These nodal sets depend closely on the geometry of the vibrating object, and a natural question to study is how the object's shape influences the nodal sets. The nodal sets also define the \textit{nodal domains} of the eigenfunctions, which are the connected sets of points that do vibrate and are separated by the nodal sets. 

There are relations between the mode of vibration and the number of nodal domains. For instance, considering standing waves on a one-dimensional string, or more formally solutions to Sturm-Liouville problems on an interval, one sees that the $k^\text{th}$ mode will have $k$ nodal domains, and a nodal set consisting of $k-1$ fixed interior points. In particular, in these one-dimensional cases, the index of the mode, $k$, is equal to the number of nodal domains.

In higher dimensions, the relation between the index of the mode and its number of nodal domains becomes much more delicate. An early advancement along these lines was Courant's Nodal Domain Theorem, which states that, in any dimension, the $k^{\text{th}}$ eigenfunction of the Dirichlet Laplacian will have at most $k$ nodal domains \cite{courant}. Very often it will have less than $k$ nodal domains, and in fact using a construction of Stern, there is an infinite sequence of eigenfunctions on the square with exactly $2$ nodal domains (see Theorem 4.1 in \cite{berard2015dirichlet}). Whenever the Dirichlet mode has exactly $k$ nodal domains, we call the corresponding eigenfunction \textit{Courant Sharp}. In two or more dimensions there are only finitely many Courant Sharp modes, as shown by Pleijel  \cite{P56} in two dimensions, and B\'erard and Meyer \cite{berard1982courant} for general $n$-dimensional Riemannian manifolds, in contrast with the one-dimensional string, which has infinitely many. For simple domains, there is a complete catalogue of all Courant Sharp modes (for example, for the Dirichlet unit square, see \cite{berard2015dirichlet}).  Other examples focused on finding the Courant Sharp modes in simple geometries include \cite{berard2016courant,berard2020courant,helffer2010spectral,helffer2016nodal,lena2015courant}, as well as the survey \cite{bonnaillie2015nodal} and references therein. However, these examples are all highly symmetric, and their analysis relies on explicit and implicit formulas of the eigenfunctions and eigenvalues, often via a separation of variables. Beyond this, however, little is known about what makes a mode Courant Sharp, and the study of Courant Sharp eigenfunctions on a general domain is an active area of research. In particular, the nodal domains of eigenfunctions on a manifold (with or without boundary) or a graph play a role in spectral partitioning into sub-domains.  We refer the reader to \cite{helffer2010spectralsurv} and the references therein for an in depth discussion about spectral partitioning.  
	
Along a different direction of study, Grieser and Jerison developed in a series of papers a detailed description of the first and second Dirichlet eigenfunctions on convex domains of large eccentricity (\cite{grieser1996asymptotics}, \cite{grieser1998size}), along with techniques for performing a precise analysis of the eigenfunction on an approximately rectangular domain (\cite{grieser2009asymptotics}). These works were extended by the first and third authors with Canzani (\cite{BCM20}), to study the nodal set of the Laplacian on curvilinear rectangles with Dirichlet boundary conditions using similar techniques. Key results from \cite{BCM20} are bounds on the slope and curvature of the nodal line for the second Dirichlet eigenfunction.
	
\begin{figure}
\begin{tabular}{cc}
  \includegraphics[width=80mm,trim={1cm 1cm 1cm 1cm},clip]{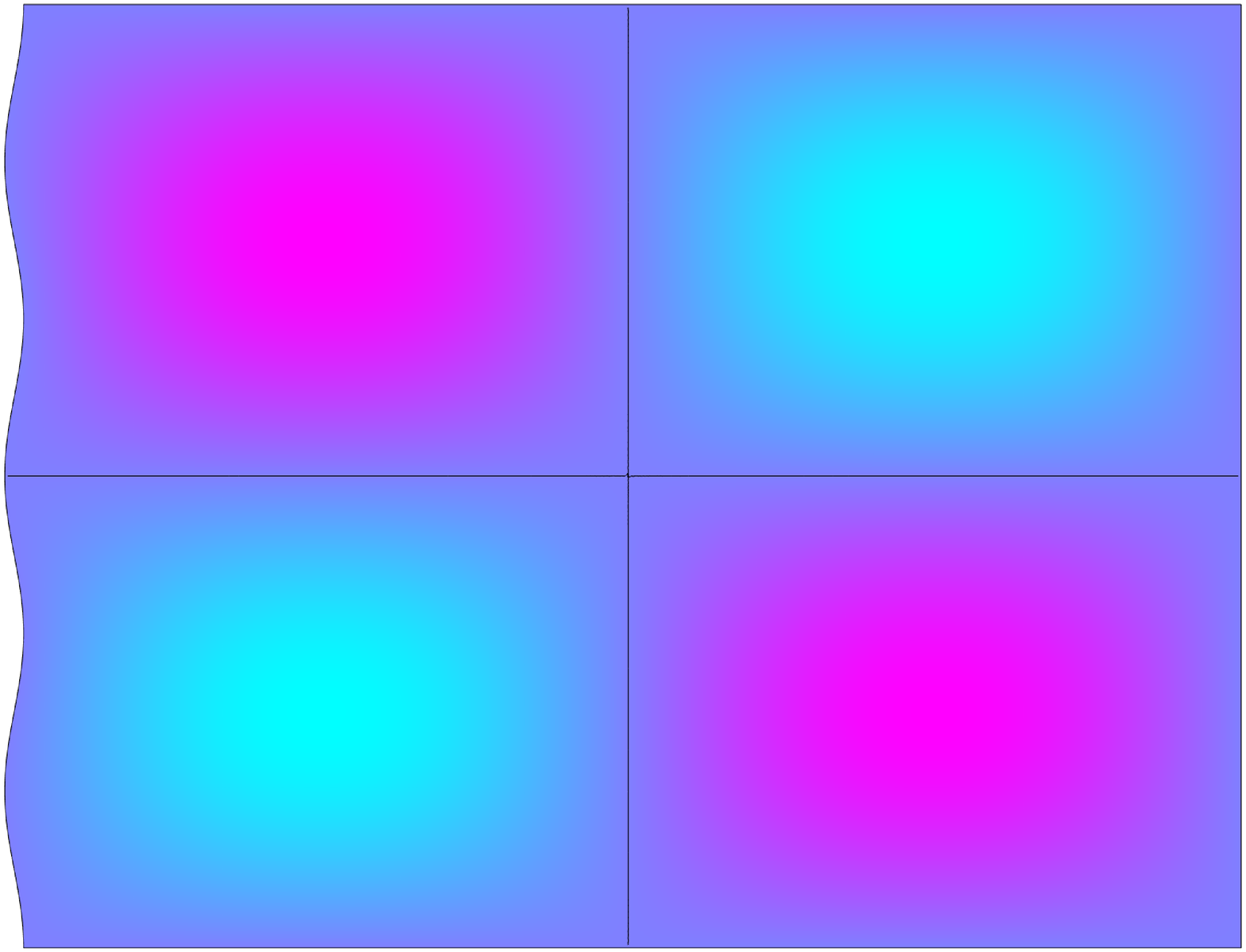} &   \includegraphics[width=62mm,trim={1cm 1cm 1cm 1cm},clip]{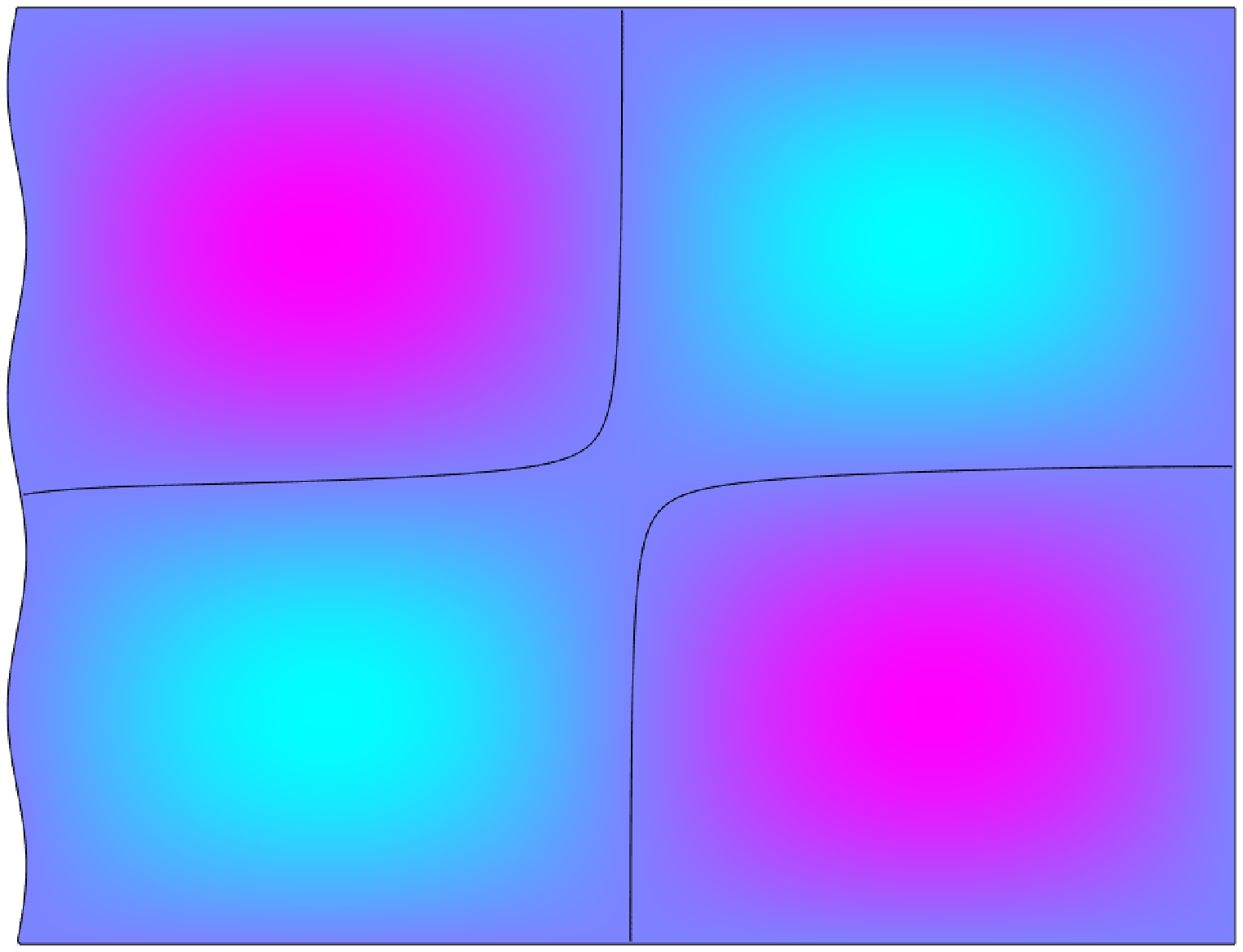} \\
(i)  & (ii)
\end{tabular}
\caption{Eigenfunctions on perturbed rectangles: (i), symmetric perturbation, aspect ratio $N = \sqrt{5/3}$, where the mode of study is Courant Sharp; (ii) asymmetric perturbation, $N = \sqrt{5/3}$.} 
\end{figure}	
	
The question explored here is a combination of the two earlier ideas - we seek to analyze how domain perturbation can decrease the number of nodal domains compared to the corresponding eigenfunction on the unperturbed domain.  That interior crossing points are unstable to perturbations has been known since the work of Uhlenbeck \cite{U76}, but here we hope to quantify that instability further. In particular, for a given rectangle of aspect ratio $N$, we consider the first eigenfunction with a crossing in its nodal set that partitions the rectangle into four nodal domains. We then perturb the boundary of the domain by an amount $\eta$ and study whether the crossing opens and if so, the scale of the opening. We note that the unperturbed eigenfunction of the rectangle is only the fourth eigenfunction for $1 \leq N \leq \sqrt{5/3}$, and so this unperturbed eigenfunction is not Courant Sharp for larger values of $N$. However, we expect the analysis to transfer over well to crossings that are present in the nodal sets of Courant Sharp eigenfunctions.  The work \cite{bogomolny2002percolation} and later \cite{beliaev2013bogomolny}  study the statistical likelihood of openings occurring at internal crossings of eigenfunctions due to perturbation using percolation theory, in order to make predictions about the asymptotic number of nodal domains for generic domains and eigenfunctions.

We make use of the techniques presented in \cite{BCM20} and \cite{grieser1996asymptotics}, including the use of the adiabatic ansatz, which is an approximate separation of variables as applied to our domain. The ansatz is a useful one for near-rectangular domains and has, for instance, been used in the study of mathematical billiards \cite{hillairet2012nonconcentration}. Using an approximate separation of variables to reduce to an associated ordinary differential operator has also been used in \cite{freitas-krej}, \cite{krej-tusek} to study eigenfunctions of the Dirichlet Laplacian on thin tubular neighborhoods of a fixed hypersurface, and in particular the structure of their nodal sets. This same idea also appears in  \cite{bor-frei}, \cite{fried-sol1}, \cite{fried-sol2} to establish asymptotics of the eigenvalues, eigenfunctions, and resolvent of the Dirichlet Laplacian in sequences of thin domains with increasing aspect ratios. The stability of properties of the nodal set is also studied in \cite{muk-saha21},\cite{muk-saha20}, where they consider low energy eigenfunctions on compact manifolds under small perturbations of the metric. They study the topology of the nodal sets, and in particular the stability of the Payne property, concerning whether the nodal set of the second eigenfunction touches the boundary.  Our results will provide both quantitative and qualitative control on the structure of the nodal set.

Numerical experiments reveal rich behavior of the eigenfunctions for differing boundary perturbations. 
For instance, as shown in Figure 1 (i-ii), the size of the opening is impacted by symmetry in the deformation. In (i), with a symmetric perturbation, there is no opening as the eigenfunction is odd with respect to the horizontal central line, while in (ii) an opening connects the two positive components of the eigenfunction with respect to our chosen sign orientation as we will fix below. 
Indeed, our results allow us to predict these properties for a given boundary perturbation and aspect ratio. Interestingly, we will see that the nodal set opening scale is independent of aspect ratio $N$, only depending on the perturbation size $\eta.$ This is in contrast with the prior works describing the nodal line in \cite{BCM20,grieser1996asymptotics,grieser1998size,grieser2009asymptotics} where the features of the nodal set depend on the aspect ratio. These prior works focus on eigenmodes with only vertical nodal lines in their nodal sets, while here, we have portions of the nodal set that run horizontally and with length comparable to the aspect ratio $N$. We attribute the lack of dependence of the size of the nodal set opening on $N$ to this fact. With these pictures in mind, we pay close attention to how geometric properties of the domain, such as symmetry in the perturbation and its relative size compared to the rectangle, influence the nodal set opening and the characteristics of the nodal set curves.

\subsection{Domain and Ansatz Description}

We start by defining the family of domains, $\Omega$, that are perturbations of the rectangle $[0,N]\times [0,1]$ for $N > 1$. Let $\phi_L(y)$ be a curve specifying the left side of $\Omega$, and write
\begin{equation}
\Omega = \{(x,y)\in \mathbb{R}^2 \,: \, y \in [0,1], \, \eta\phi_L(y) \leq x \leq N\}.
\end{equation} 
We require $\phi_L \in C^5([0,1]; \mathbb{R})$ with $\phi_L(0) = \phi_L(1) = 0$, and to ensure that the domain perturbation is small, we require that for all $y\in[0,1]$,
\begin{equation} \label{boundary-curve-specification}
-1 \leq \phi_L(y) \leq 0, 
\qquad |\phi_L^{(j)}(y)| \leq 1, \text{ for }1 \leq j\leq 5.
\end{equation}
Throughout, $\eta>0$ is a small constant and note that taking $\eta = 0$ recovers the rectangle $\Omega = [0,N]\times [0,1]$. We highlight two quantities: $N$, the aspect ratio on the unperturbed rectangle, and the ratio $\frac{\eta}{N}$, which describes the relative size of the side perturbation.  The estimates we have away from the left boundary in fact require far fewer than $5$ derivatives of $\phi_L$ to be bounded.  The regularity restriction we set here has more to do with carefully controlling the intersection of the nodal set with the left boundary.  We do not expect this regularity restriction on $\phi_L$ to be sharp for our results to hold however, but this is what we require for the techniques that we apply here.

We proceed to identify our eigenfunction of interest, $v$. This eigenfunction depends on $N$, $\eta$ and $\phi_L$, and in our results we will fix $\phi_L$, and take $N$ to be sufficiently large and $\eta$ sufficiently small. To start, consider the Laplacian with Dirichlet boundary conditions on the unperturbed rectangle, $[0, N] \times [0,1]$. That is, consider solutions to 
\begin{equation}
\begin{cases}
\Delta \varphi + \lambda \varphi = 0 \text{ in }[0, N] \times [0,1] \\
\varphi  = 0 \text{ on } \partial([0, N] \times [0,1]).
\end{cases}
\end{equation}
Up to a normalizing constant, solutions to this system are given by $\varphi_{n,m}(x,y) = \sin(\frac{n\pi}{N}x)\sin(m\pi y),$ with corresponding eigenvalues $\lambda_{n,m} = \pi^2(\frac{n^2}{N^2} + m^2)$.

We focus on $\varphi_{2,2}(x,y) =\sin(\frac{2\pi}{N}x)\sin(2\pi y)$, with eigenvalue $\lambda_{2,2} = 4\pi^2(\frac{1}{N^2} + 1)$. We then take $v(x,y)$ to be a Dirichlet eigenfunction for $\Omega$, in the same position of the spectrum as $\varphi_{2,2}$. See Assumption \ref{assum:simple} below for an assumption on $N$ that ensures that this eigenvalue is simple for all $\eta>0$ sufficiently small. This will ensure that for fixed $N$ and $\phi_L$,  $v(x,y) = v(x,y;\eta)$ depends smoothly on $\eta$, with $v(x,y;0) = \varphi_{2,2}(x,y)$. Letting $\mu$ be the eigenvalue of $v$, the eigenfunction $v$ then solves 
\begin{equation}
\begin{cases}
\Delta v + \mu v = 0 \text{ in }\Omega \\
v|_{\partial \Omega} = 0.
\end{cases}
\end{equation}
We will normalize $v$ to have $L^{\infty}$-norm equal to $1$, and as for $\varphi_{2,2}(x,y)$ to be negative at $(\tfrac{N}{4},\tfrac{3}{4})$, $(\tfrac{3N}{4},\tfrac{1}{4})$ and positive at $(\tfrac{N}{4},\tfrac{1}{4})$, $(\tfrac{3N}{4},\tfrac{3}{4})$.

Observe that the nodal set of $\varphi_{2,2}$ can be written as $\Gamma_{2,2} = \left([0,N] \times \{\frac{1}{2}\}\right) \cup  \left(\{\frac{N}{2}\} \times [0,1]\right)$. In other words, the nodal set of $\varphi_{2,2}$ consists of two lines that intersect in the center of the rectangle and partition it into four connected components. We are interested in determining whether the nodal set of $v$ maintains a crossing in its nodal set, and if it does not, how the size of the opening in the nodal set depends on $\eta$ and $\phi_L$. Moreover, we seek information in the orientation of the opening (that is, whether the connected component in the center runs from the lower-left to upper-right ends of $\Omega$ or the upper-left to lower-right).

In order for our choice of eigenfunction $v$ to be well-defined,  and for its nodal set to approach that of $\varphi_{2,2}$ as $\eta$ tends to $0$, we need to ensure that eigenvalue $\lambda_{2,2} = 4\pi^2\left(\tfrac{1}{N^2}+1\right)$ is simple. This eigenvalue fails to be simple when an eigenfunction $\varphi_{k,1}$, with eigenvalue $\lambda_{k,1}=\pi^2\left(\tfrac{k^2}{N^2}+1\right)$, leads to a repeated eigenvalue. To ensure that the eigenvalue $\mu$ is a simple eigenvalue of $\Omega$ for all $\eta>0$ sufficiently small, we therefore make the following assumption.
\begin{assum} \label{assum:simple}
The value of $N$ satisfies
\begin{align*}
    \min_{k}\left|4\left(\tfrac{1}{N^2}+1\right) - \left(\tfrac{k^2}{N^2}+1\right) \right| =     \min_{k}\left|3 + \tfrac{4}{N^2} - \tfrac{k^2}{N^2} \right| \neq 0.
\end{align*}
\end{assum}
In other words, this assumption implies that $N$ is not equal to a sequence of values $N_3 = \sqrt{5/3}$, $N_4=2, \ldots$, where the above quantity vanishes for $k=3,4,\ldots$.

Before we state our main results, it will be helpful to describe our eigenfunction decomposition. For fixed $x\in[0,N]$, we apply a Fourier series in $y\in[0,1]$ to write
\begin{equation} \label{ansatz}
v(x,y) =\sum_{k=1}^{\infty}v_k(x)\sin(k\pi y) =  v_1(x) \sin(\pi y) + v_2(x) \sin(2\pi y) + E(x,y),
\end{equation}
 where for $k \geq 0$,
\begin{equation} \label{mode-inner-product}
v_k(x) = 2 \int_0^1 v(x,y) \sin(k\pi y)\, dy.
\end{equation}
The term $E(x,y) = \sum_{k \geq 3} v_k(x) \sin(k \pi y)$ contains all of the higher modes, and we will see that $E(x,y)$ is bounded by a multiple of $\eta/N$ in $\Omega$, and is exponentially small in $N$ away from the left and right boundaries of $\Omega$. When $\eta=0$, $v_1(x)$ and $E(x,y)$ are identically zero, and we will see that for $\eta>0$ sufficiently small, $v_2(x) \sin(2\pi y)$ remains the main contribution to $v$, and that the nodal set of $v_2(x) \sin(2\pi y)$ maintains a crossing. It will be the first Fourier mode $v_1(x)$ that controls the potential opening of the nodal set of $v$. This mode satisfies the ODE
\begin{align}
v_1''(x) + \mu^2_1 v_1(x) = 0,
\end{align}
where the coefficient $\mu_1$ satisfies
\begin{align*}
    \mu_1^2 = \mu - \pi^2.
\end{align*}
When $\eta=0$, $\mu_1^2$ can be explicitly computed and is given by
\begin{align} \label{eqn:mu1-zero}
    \mu_1(0)^2 = \lambda_{2,2} - \pi^2 = \pi^2\left(3 + \tfrac{4}{N^2}\right).
\end{align}
Note that Assumption \ref{assum:simple} ensures that $\mu_1(0)$ is bounded away from an integer multiple of $\tfrac{\pi}{N}$.

The following proposition provides the bounds and properties stemming from the decomposition made in (\ref{ansatz}), (\ref{mode-inner-product}) that we will require. We will see that the modes $v_1(x)$ and $v_2(x)$ are given by trigonometric functions with coefficients depending on their values at $x=0$ and $x=N$. Therefore, we extend their definitions to $-\eta<x<0$ so that the error $E(x,y)$ is defined for all $(x,y)\in\Omega$. 

\begin{prop}\label{bounds-prop}
There exist constants $c$, $C$, and $\eta_0>0$ such that for all $N > 4$ and $0<\eta < \eta_0$, the following holds:
\begin{enumerate}

    \item\label{v1-prop} The first Fourier mode $v_1(x)$ satisfies the estimates
    \begin{align*}
        \big|v_1^{(j)}(x)\big| \leq C\eta/N \text{ for } 0\leq j\leq3.
    \end{align*}
    Moreover, under Assumption \ref{assum:simple}, 
    \begin{align*}
        C^{-1}|v_1(0)| \leq |v_1(\tfrac{N}{2})| \leq C|v_1(0)|.
    \end{align*}
    
    \item\label{v2-prop} The second Fourier mode $v_2(x)$ satisfies the estimates
    \begin{align*}
    \left|v_2^{(j)}(x) - \tfrac{d^j}{dx^j}\left(\sin\left(\tfrac{2\pi}{N}x\right)\right)\right| \leq C\eta N^{-j-1} \text{ for } 0\leq j\leq 3.
    \end{align*}
    In particular, there exists a unique $x^*\in\big(\tfrac{N}{2}-C\eta,\tfrac{N}{2}+C\eta\big)$ such that $v_2(x^*) = 0$.
    
    \item\label{E-prop} The error $E(x,y)$ satisfies
    \begin{align*}
        |E(x,y)| + |\nabla E(x,y)| \leq C\eta/N \text{ for all } (x,y)\in\Omega,
    \end{align*}
    together with estimates away from the corners of $\Omega$ given by
    \begin{align*}
        \sup_{ y\in[1/10,9/10]}\left|\nabla^j E(x,y) \right| \leq C\eta/N, \quad \sup_{x\in [N/4,3N/4]} \Big| \nabla^j E(x,y) \Big| \leq C\eta e^{-cN} \text{ for } 0 \leq j \leq 3.
    \end{align*}
    
\end{enumerate}
\end{prop}
We will prove this proposition in Section \ref{sec:ansatz}, using the decomposition of $v$ given in \eqref{ansatz}, together with the estimate on the Fourier coefficients $v_k(0)$ given in the following proposition:
\begin{prop} \label{hadamard-prop}
There exists a constant $C$ such that for all $k \geq 1$, 
\begin{equation} \label{boundary-integral}
\left|v_k(0) - \frac{4\pi\eta}{N} \int_0^1 \phi_L(y) \sin(2\pi y)\sin(k \pi y) dy\right| \leq C(\eta^{2}/N+k^2\eta^3/N).
\end{equation}
\end{prop}
We will also prove this proposition in Section \ref{sec:ansatz}, using a Hadamard variation argument. 
When $v_1(0)\neq0$, we will see that the nodal set of $v$ no longer has a crossing, and we will obtain a detailed description of how the splitting of the nodal set occurs (see Theorem \ref{thm:nodal-gap-estimate} below). Motivated by Proposition \ref{hadamard-prop}, we therefore place the following generic assumption on the left boundary function $\phi_L(y)$.
\begin{assum} \label{assum:phi}
The function $\phi_L(y)$ satisfies
\begin{align*}
   \int_0^1 \phi_L(y) \sin(2\pi y)\sin(\pi y) \,dy \neq 0.
\end{align*}
\end{assum}

By Proposition \ref{hadamard-prop}, this assumption ensures that for $\eta>0$ sufficiently small, the first Fourier mode satisfies
\begin{align} \label{eqn:phi1}
   C^{-1}\eta/N \leq |v_1(0)| \leq C\eta/N.
\end{align}
From now on, we will assume that the two assumptions hold, and all constants appearing in our estimates will depend on a lower bound on the minimum and integral in Assumptions \ref{assum:simple} and \ref{assum:phi} respectively. We will use the notation $A\sim B$, when $|A/B|$ is bounded from above and below by constants depending only on these lower bounds. Our first main theorem concerns the location of the nodal set of $v$.
\begin{theorem} \label{thm:nodal-gap-estimate}
There exist constants $\eta_0>0$ and $N_0>1$, $A_0>0$, such that for all $0<\eta<\eta_0$, $N>N_0$, the nodal set of $v$ has the following properties: 
\begin{enumerate}
    \item[i)] It consists of two curves which are a distance $\sim\sqrt{\eta}$ apart, and separates $\Omega$ into three connected components.
    \item[ii)] There exists a strip of width $\sim\sqrt{\eta}$ passing through the center $\left(\tfrac{N}{2},\tfrac{1}{2}\right)$ in the direction of either $[1,1]^{T}$ or $[1,-1]^{T}$ which is disjoint from the nodal set.
    \item[iii)] For $x\in[\tfrac{N}{2}-\tfrac{1}{10},\tfrac{N}{2}+\tfrac{1}{10}]$, the eigenfunction $v$ does not vanish whenever $(x,y)$ satisfies
    \begin{align*}
        \left|x-\tfrac{N}{2}\right|\left|y- \tfrac{1}{2}\right| \geq A_0\eta.
    \end{align*}
    \item[iv)] For $x\notin[\tfrac{N}{2}-\tfrac{1}{10},\tfrac{N}{2}+\tfrac{1}{10}]$, the eigenfunction $v$ does not vanish whenever $(x,y)$ satisfies
    \begin{align*}
        \left|y- \tfrac{1}{2}\right| \geq \frac{A_0\eta}{N\left|\sin\left(\tfrac{2\pi x}{N}\right)\right| + 1}.
    \end{align*}
\end{enumerate} 
\end{theorem}
\begin{Rem}
The theorem says that the nodal set of $v$ resembles that of one of the hyperbolas
\begin{align*}
\left(x-\tfrac{N}{2}\right)\left(y-\tfrac{1}{2}\right) = \pm\eta. 
\end{align*}
In the course of the proof of Theorem \ref{thm:nodal-gap-estimate}, we will see which of these two hyperbolas the nodal set resembles, and that the direction in ii) is determined by the sign at the center of the first Fourier mode $v_1(x)$.  In particular, we observe that the number of nodal domains decreases from four to three under perturbations satisfying Assumptions \ref{assum:simple} and \ref{assum:phi}.  We expect that similar behavior occurs near any interior crossing  under generic perturbations for higher oscillatory modes on the model rectangle.  This includes any Courant sharp eigenfunctions on the model rectangle.
\end{Rem}
Our second main theorem concerns the regularity of the nodal set of $v$.
\begin{theorem} \label{thm:regularity}
There exist constants $\eta_0>0$ and $N_0>1$, $C_0>0$, such that for all $0<\eta<\eta_0$, $N>N_0$, the nodal set of $v$ has the following properties:
\begin{enumerate}
    \item[i)] With respect to the direction given in Theorem \ref{thm:nodal-gap-estimate} ii), the nodal set of $v$ can be written as the graphs of two smooth functions with bounded derivatives.
    \item[ii)] Outside of a disc of radius $\sim \sqrt{\eta}$ centered at $(\tfrac{N}{2},\tfrac{1}{2})$, the nodal set can be parameterized as the graphs of functions $x = g(y)$, and $y = h(x)$, with 
    \begin{align*}
        |g'(y)| \leq \frac{C_0\eta}{|y-\tfrac{1}{2}|^2},
    \end{align*}
    and
    $$
    |h'(x)| \leq
    \begin{cases}\frac{C_0\eta}{N\left|\sin\left(\tfrac{2\pi x}{N}\right)\right| + 1} & \emph{ if } \left|x-\tfrac{N}{2}\right| \geq \tfrac{1}{10} \\
    \frac{C_0\eta}{\left|x-\tfrac{N}{2}\right|^2} & \emph{ if } \left|x-\tfrac{N}{2}\right| < \tfrac{1}{10}.
    \end{cases}
    $$
    \item[iii)] The nodal set intersects the boundary $\pa\Omega$ orthogonally at precisely $4$ points.
\end{enumerate}
\end{theorem}
\begin{Rem} \label{Nrem}
In Theorems \ref{thm:nodal-gap-estimate} and \ref{thm:regularity}, taking the length scale $N$ to be sufficiently large is only used to guarantee the structure of the opening of the nodal set near the center. Combining \eqref{eqn:phi1} with Proposition \ref{bounds-prop} \eqref{v1-prop} ensures that $|v_1(\tfrac{N}{2})|\sim \eta/N$. We only require $N$ to be large so that the exponentially small bound of  $C\eta e^{-cN}$ on $|E(\tfrac{N}{2},\tfrac{1}{2})|$ from Proposition \ref{bounds-prop} \eqref{E-prop} is small compared to $\eta/N$, as this then implies that $|v(\tfrac{N}{2},\tfrac{1}{2})| = |v_1(\tfrac{N}{2}) + E(\tfrac{N}{2},\tfrac{1}{2})|$ is also comparable to $\eta/N$. It is this estimate $|v(\tfrac{N}{2},\tfrac{1}{2})|\sim \eta/N$ that leads to the opening of the nodal set of size $\sqrt{\eta}$ at the center. For any fixed $N>1$, our proof will provide an upper bound of size $\sqrt{\eta}$ on the opening at the center. Moreover, if Assumptions \ref{assum:simple} and \ref{assum:phi} hold, together with the estimate $|v(\tfrac{N}{2},\tfrac{1}{2})|\sim \eta/N$, then all of the results of Theorem \ref{thm:nodal-gap-estimate} and \ref{thm:regularity} hold for any fixed $N>1$ and all $\eta>0$ sufficiently small. In particular, in this case the eigenfunction $v$ has three nodal domains for all $\eta>0$.
\end{Rem}

 The structure of the rest of the paper is as follows. In Section \ref{sec:center}, we study the nodal set of $v$ in a disc of radius comparable to $\eta^{2/5}$, 
 centered at $(\tfrac{N}{2},\tfrac{1}{2})$. In particular, by proving a sufficiently sharp approximation of the nodal set by a hyperbola, we establish the opening of the crossing for $\eta>0$. In Section \ref{sec:edge}, we prove that away from this disc centered at $(\tfrac{N}{2},\tfrac{1}{2})$, the nodal set of $v$ is given by $4$ almost \textit{flat} curves, and control the deviation of these curves from straight lines as they move towards the boundary of $\Omega$. Then, in Section \ref{sec:perpendicular}, we will prove estimates on the perpendicularity of the nodal set of $v$ at the $4$ points where it meets $\pa\Omega$. In Section \ref{sec:ansatz}, we establish Propositions \ref{bounds-prop} and \ref{hadamard-prop}.  Finally, in Section \ref{sec:numerics} we show some numerics using Matlab to illustrate our results and the role of Assumptions \ref{assum:simple} and \ref{assum:phi}.

\subsection*{Acknowledgements} 

The authors are grateful to Yaiza Canzani and Stefan Steinerberger for helpful conversations about related problems.  T.B. was supported through NSF Grant DMS-2042654.  J.L.M.
acknowledges support from the NSF through NSF grant DMS-1909035. M.G. received support through the UNC SURF program for Undergraduate Research.

\section{Properties of the Nodal Set Opening} \label{sec:center}

In this section, we will study the nodal set of the eigenfunction $v$ in a disc centered at $(\tfrac{N}{2},\tfrac{1}{2})$. By choosing the radius of this disc appropriately, we will show that inside this disc the nodal set of $v$ is sufficiently close to that of a hyperbola and consists of two smooth curves separated by a distance comparable to $\sqrt{\eta}$. In order for $v$ to be closely approximated by this model hyperbola, in this section we will work in the disc $D_{\eta^{2/5}}(\tfrac{N}{2},\tfrac{1}{2})$, centered at $(\tfrac{N}{2},\tfrac{1}{2})$ and of radius $\eta^{2/5}$. The hyperbola that the eigenfunction $v$ is closely approximated by is the following.
\begin{Def} \label{defn:hyperbola}
With $v_1(x)$, $v_2(x)$ as in \eqref{mode-inner-product} and $\mu_1^2 = \mu-\pi^2$, we define the quadratic polynomial $P(x,y)$ by
\begin{align*}
    P(x,y) = v(\tfrac{N}{2},\tfrac{1}{2}) + \nabla v(\tfrac{N}{2},\tfrac{1}{2}) \left(x-\tfrac{N}{2},y-\tfrac{1}{2}\right)^{T} + \tfrac{1}{2}\left(x-\tfrac{N}{2},y-\tfrac{1}{2}\right)\mathbf{H}\left(\tfrac{N}{2}, \tfrac{1}{2}\right) \left(x - \tfrac{N}{2}, y - \tfrac{1}{2}\right)^T.
\end{align*}
Here $\mathbf{H}\left(\tfrac{N}{2}, \tfrac{1}{2}\right)$ is the Hessian matrix
\begin{align*}
    \mathbf{H}\left(\tfrac{N}{2}, \tfrac{1}{2}\right) = \begin{pmatrix}
-\mu_1^2 v_1\left(\tfrac{N}{2}\right) & -2\pi v'_2\left(\tfrac{N}{2}\right)\\
-2\pi v'_2\left(\tfrac{N}{2}\right) & -\pi^2 v_1\left(\tfrac{N}{2}\right)
\end{pmatrix}.
\end{align*}
\end{Def}
To determine the behavior of the nodal set of $v$ in $D_{\eta^{2/5}}(\tfrac{N}{2},\tfrac{1}{2})$, we will first study the nodal set of $P(x,y)$, and then use the following proposition to translate these properties of the nodal set of $v$.
\begin{prop} \label{prop:hyperbola-approx}
There exist constants $\eta_0>0$ and $C_1$ such that for all $N>1$ and $0<\eta<\eta_0$, the remainder $R(x,y) = v(x,y) - P(x,y)$ satisfies
\begin{align*}
    \sup_{(x,y)\in D_{2\eta^{2/5}}\left(\tfrac{N}{2},\tfrac{1}{2}\right)}|R(x,y)| + \eta^{2/5}|\nabla R(x,y)| \leq C_1 \eta^{8/5}/N.
\end{align*}
\end{prop}
The nodal set of $P(x,y)$ has the following properties.
\begin{prop} \label{prop:hyperbola}
There exist constants $\eta_0>0$, $N_0>1$ and $C_2$ such that for all $0<\eta<\eta_0$, $N>N_0$, the nodal set of $P(x,y)$ is a hyperbola satisfying the following properties:
\begin{enumerate}

\item[i)] The center of the hyperbola is at the point $(x_c,y_c)$, with $\left|(x_c,y_c)-(\tfrac{N}{2},\tfrac{1}{2})\right|\leq C_2\eta.$ Moreover, $P(x,y)$ can be written as
\begin{align*}
  P_0 -\tfrac{1}{2}\mu_1^2 v_1\left(\tfrac{N}{2}\right)(x - x_c)^2 -2\pi v'_2\left(\tfrac{N}{2}\right) (x-x_c)(y-y_c) - \tfrac{1}{2}\pi^2 v_1\left(\tfrac{N}{2}\right) (y - y_c)^2,
\end{align*}
for a constant $P_0$ satisfying $|P_0|\sim \eta/N$.

\item[ii)] The distance between the vertices of the hyperbola is comparable to $\sqrt{\eta}$.

\item[iii)] The principal axis of the hyperbola (passing through its vertices) makes an angle $\varphi$ with the positive $x$-axis with either $|\varphi - \tfrac{\pi}{4}| < C_2\eta$ or $|\varphi + \tfrac{\pi}{4}|<C_2\eta$. This orientation will be determined by the sign of $v_2'(\tfrac{N}{2})/v_1(\tfrac{N}{2})$, with $v_1(x)$, $v_2(x)$ from \eqref{ansatz} and \eqref{mode-inner-product}. That is, $\varphi \approx \tfrac{\pi}{4}$ for $v_2'(\tfrac{N}{2})/v_1(\tfrac{N}{2})>0$ and $\varphi\approx-\tfrac{\pi}{4}$ for $v_2'(\tfrac{N}{2})/v_1(\tfrac{N}{2})<0$. 

\end{enumerate}
\end{prop}
Before proving these two propositions, let us first use them to establish the parts of Theorems  \ref{thm:nodal-gap-estimate} and \ref{thm:regularity} in the disc $D_{\eta^{2/5}}(\tfrac{N}{2},\tfrac{1}{2})$. By the bounds on $v_1(x)$ and $v_2(x)$ given in Proposition \ref{bounds-prop} \eqref{v1-prop} and \eqref{v2-prop}, $|v_1(\tfrac{N}{2})|\leq C\eta/N$ and $|v_2'(\tfrac{N}{2})|\sim 1/N$. Therefore, by introducing new coordinates $(x',y')$, centered at $(x_c,y_c)$ and rotated by an angle $\varphi$, and using Proposition \ref{prop:hyperbola} i), we can write $v$ as
\begin{align} \label{eqn:hyperbola1}
  v(x',y') = P_0 - \frac{(x')^2}{A^2} + \frac{(y')^2}{B^2} + R'(x',y'),
\end{align}
with $|A|, |B|\sim \sqrt{N}$, and $|R'(x',y')| + \eta^{2/5}|\nabla R'(x',y')| \leq C\eta^{8/5}/N$ for $|(x',y')| \leq 2\eta^{2/5}$.

To prove the parts of Theorem \ref{thm:nodal-gap-estimate} (i), (ii) and Theorem \ref{thm:regularity} (i) inside the disc $D_{\eta^{2/5}}(\tfrac{N}{2},\tfrac{1}{2})$ near the center, we first exclude the remainder term from \eqref{eqn:hyperbola1} to obtain a preliminary estimate on the nodal set separation and regularity, and then correct for the remainder. For ease of notation we drop the primes in $(x',y')$ and so work with $(x,y)$ in a disc of radius $2\eta^{2/5}$ centered at the origin. Let $(x, y)$ be a solution of
\begin{equation}  \label{eqn:model-nodal-set}
    \frac{x^2}{A^2} - \frac{y^2}{B^2} = P_0,
\end{equation}
with $|(x,y)| \leq \eta^{2/5},$ where we assume without loss of generality that $P_0 > 0$. This is the equation of a hyperbola, and we write the right branch of it as $x := x(y) >0,$ giving
\begin{equation} \label{eqn:model-hyperbola-graph-over-y}
    x(y) = \frac{A}{B} \sqrt{B^2 P_0 + y^2}.
\end{equation}
Let $(\tilde{x}, y)$ denote a solution to
\begin{equation}  \label{eqn:true-nodal-set}
    \frac{\tilde{x}^2}{A^2} - \frac{y^2}{B^2} = P_0 + R'(\tilde{x},y), 
\end{equation}
so that by \eqref{eqn:hyperbola1}, $(\tilde{x},y)$ is on the nodal set of $v$.

We will look for a solution of \eqref{eqn:true-nodal-set} of the form $\tilde{x}(y) = x(y) + w(y)$ for some perturbation $w = w(y)$. The model hyperbola equation \eqref{eqn:model-nodal-set} implies that $|x(y)| \geq c\sqrt{\eta}.$ To obtain bounds on the opening distance for the true nodal set \eqref{eqn:true-nodal-set}, we will prove bounds on the size and regularity of the perturbation $w$. We first prove that such a solution of \eqref{eqn:true-nodal-set} exists for each fixed $y$ with $|(x(y),y)| \leq \eta^{2/5}$, and then obtain crude bounds on $w$ ensuring that $|(\tilde{x},y)| \leq 2\eta^{2/5}$. We then use the resulting bounds on $R'(\tilde{x},y)$ for $(\tilde{x},y)$ in this disc to obtain tighter bounds. The left branch of the hyperbola can be treated in an identical way.

Substituting $\tilde{x}(y) := x(y) + w(y)$ into \eqref{eqn:true-nodal-set} and using \eqref{eqn:model-nodal-set} gives
\begin{align*}
    \frac{w^2}{A^2} + \frac{2xw}{A^2} = R'(x+w, y).
\end{align*}
Fix $y_0$ with $|(x(y_0), y_0)| \leq \eta^{2/5},$ and set
\begin{align*}
    f(y_0,w) = f_{y_0}(w) = \frac{w^2}{A^2} + \frac{2xw}{A^2} - R'(x+w, y_0).
\end{align*}
The function $f_{y_0}(w)$ is then equal to zero precisely when $(\tilde{x},y_0)$ satisfies \eqref{eqn:true-nodal-set}. Observe then that, since $x=x(y_0) > 0$ 
\begin{align*}
    f_{y_0}(\eta^{2/5}) & = \frac{\eta^{4/5}}{A^2} + \frac{2x \eta^{2/5}}{A^2} - R'(x(y_0)+\eta^{2/5}, y_0) \\
     & \geq  \frac{\eta^{4/5}}{A^2}  - \sup_{|(x,y)|\leq \eta^{2/5}}|R'(x+\eta^{2/5}, y)|
     \geq \frac{\eta^{4/5}}{A^2} - C\frac{\eta^{8/5}}{N} > 0,
\end{align*}
for $\eta$ sufficiently small. Here we have applied the estimates on $R'$ and used $|A| \sim \sqrt{N}$ from \eqref{eqn:hyperbola1}. Similarly, using $x = x(y_0) \geq c\sqrt{\eta}$, observe that
\begin{align*}
    f_{y_0}(-\tfrac{c}{2}\sqrt{\eta}) &= \frac{c^2 \eta}{4A^2} - \frac{cx\sqrt{\eta}}{A^2} + R'(x - \tfrac{c}{2}\sqrt{\eta}, y_0) 
    \leq \frac{c^2 \eta}{4A^2} - \frac{c^2\eta}{A^2} + C\frac{\eta^{8/5}}{N} < 0
\end{align*}
for $\eta$ sufficiently small. Therefore, for each $y_0$ there exists some $w \in (-\tfrac{c}{2}\sqrt{\eta},\eta^{2/5})$ satisfying $f_{y_0}(w) = 0,$ or equivalently satisfying \eqref{eqn:true-nodal-set}, where the constant $c$ is independent of $y_0$.  In particular, the lower bound on $w$ implies that $\tilde{x} \geq \tfrac{c}{2}\sqrt{\eta} > 0$, and the upper bound on $w$ implies that $|(\tilde{x}, y)| \leq 2\eta^{2/5}.$

To refine the bound on $w = w(y)$, subtract \eqref{eqn:model-nodal-set} from \eqref{eqn:true-nodal-set} to obtain
\begin{align*}
    \tilde{x}^2 - x^2 = w(\tilde{x} + x) = A^2 R'(\tilde{x},y).
\end{align*}
Since $\tilde{x}$ and $x$ are both positive, and $|(\tilde{x},y)| \leq 2\eta^{2/5}$, we have
\begin{align*}
     |w| = \left|\frac{A^2 R'(\tilde{x},y)}{\tilde{x} + x}\right| \leq C\frac{\eta^{8/5}}{\sqrt{\eta}} \leq C \eta^{11/10}.
\end{align*}

To study the impact of the error term $R'$ on the regularity of the nodal set, pick $y_0, w_0$ such that $(\tilde{x}(y_0), y_0)$ lies in the nodal set. Then $f(y_0, w_0) = 0$. Further, using \eqref{eqn:model-hyperbola-graph-over-y}, we have that
\begin{align*}
|\partial_w f(y_0, w_0)| &= \left|\frac{2w_0}{A^2} + \frac{2}{AB}\sqrt{P_0 B^2+y_0^2} - \partial_x R'(\tilde{x}(y_0), y_0) \right|\\
&\geq \left|\frac{2}{AB}\sqrt{P_0 B^2+y_0^2} \right| - \left|\frac{2w_0}{A^2} \right| - \left| \partial_x R'(\tilde{x}(y_0), y_0) \right|\\
&\geq \frac{2}{A}\left(\left|\sqrt{P_0}\right| - \left|\frac{w_0}{A}\right|\right) - \left| \partial_x R'(\tilde{x}(y_0), y_0) \right|.
\end{align*}
Combining the estimates on $A$, $P_0$ and $R'$ given after \eqref{eqn:hyperbola1}, with the bound $|w_0| = |w(y_0)| \leq C\eta^{11/10}$ therefore gives that $|\partial_w f(y_0, w_0)| > c\eta^{1/2}/N$ for $\eta>0$ sufficiently small. 
This gives the existence of a graph function $w(y)$ and a neighborhood $U$ of $(y_0, w_0)$ such that $f^{-1}(0)\cap U = \{(y, w(y))\,:\, y \in U\}$, with
\begin{align*}
|w'(y)| = \frac{|\partial_y f(y,w(y))|}{|\partial_w f(y,w(y)|}.
\end{align*}
We also have the upper bound
$$|\partial_y f(y,w(y))| = \left| \frac{2yw}{AB\sqrt{P_0 B^2 + y^2}} - \partial_y R'(\tilde{x}(y),y)\right|  \leq C\eta^{11/10}/N$$
for $\eta$ sufficiently small, again using the estimates $|AB| \sim N$, $y/\sqrt{P_0B^2+y^2}\leq 1$, and bounds on $|\nabla R'|$. This combined with the lower bound on $|\partial_w f(y_0, w_0)|$ implies that $|w'(y)| \leq C\eta^{6/5}$.  

The parts of Theorem \ref{thm:nodal-gap-estimate} i) and ii) in the disc $D_{\eta^{2/5}}(\tfrac{N}{2},\tfrac{1}{2})$ then follow immediately from combining the separation of the model hyperbola in \eqref{eqn:model-nodal-set} with the estimate $|w(y)| \leq C\eta^{11/10}$. From the hyperbola model, the regularity of the model nodal set $x = x(y)$ satisfies 
\begin{align*}
  |x'(y)| = \left|Ay/\left(B\sqrt{B^2 P_0 + y^2)}\right)\right| \leq C|y|/\sqrt{\eta+y^2}\leq C.  
\end{align*}
Combining this with the above estimate $|w'(y)| \leq C\eta^{6/5}$, this part of the nodal set of $v$ is the graph of a smooth function with a bounded derivative, giving the part of Theorem \ref{thm:regularity} in $D_{\eta^{2/5}}(\tfrac{N}{2},\tfrac{1}{2})$. 
\\
\\
We are left in this section to prove Propositions \ref{prop:hyperbola-approx} and \ref{prop:hyperbola}.
\begin{proof1}{Proposition \ref{prop:hyperbola-approx}}
We start by Taylor expanding $v(x,y)$ about the center $(\tfrac{N}{2},\tfrac{1}{2})$ to get
\begin{align}
v(x,y) &= v\left(\tfrac{N}{2}, \tfrac{1}{2}\right) + \nabla v\left(\tfrac{N}{2}, \tfrac{1}{2}\right) \left(x - \tfrac{N}{2}, y - \tfrac{1}{2}\right)^T \notag \\
&  \hspace{.5cm} + \tfrac{1}{2} \left(x -\tfrac{N}{2}, y - \tfrac{1}{2}\right) \mathbf{H}_2\left(\tfrac{N}{2}, \tfrac{1}{2}\right) \left(x - \tfrac{N}{2}, y - \tfrac{1}{2}\right)^T + R_3(x,y).  \label{taylor-expansion}
\end{align}
Using the decomposition of $v(x,y)$ given in \eqref{ansatz}, we can write 
\begin{equation} \label{gradient-center}
\nabla v\left(\tfrac{N}{2}, \tfrac{1}{2}\right) = \displaystyle \begin{pmatrix}
v_1'\left(\tfrac{N}{2}\right)  \\ -2\pi v_2\left(\tfrac{N}{2}\right) 
\end{pmatrix} + \begin{pmatrix}
 \partial_x E \\  \partial_y E
\end{pmatrix}
\end{equation}
and
\begin{equation} \label{hessian-center}
\mathbf{H}_2\Big(\tfrac{N}{2}, \tfrac{1}{2}\Big) = \begin{pmatrix}
\pa_x^2v(\tfrac{N}{2},\tfrac{1}{2}) & \pa_x\pa_{y}v(\tfrac{N}{2},\tfrac{1}{2})\\
\pa_{x}\pa_{y}v(\tfrac{N}{2},\tfrac{1}{2}) & \pa_y^2v(\tfrac{N}{2},\tfrac{1}{2})
\end{pmatrix} = \begin{pmatrix}
-\mu_1^2 v_1\left(\tfrac{N}{2}\right) & -2\pi v'_2\left(\tfrac{N}{2}\right)\\
-2\pi v'_2\left(\tfrac{N}{2}\right) & -\pi^2 v_1\left(\tfrac{N}{2}\right)
\end{pmatrix} + \begin{pmatrix}
\partial_{xx}E & \partial_{xy}E\\
\partial_{xy}E & \partial_{yy}E
\end{pmatrix}.
\end{equation}
Here all the partial derivatives of $E$ are evaluated at $\left(\tfrac{N}{2}, \tfrac{1}{2}\right)$. In particular, using the bounds on $E(x,y)$ from Proposition \ref{bounds-prop} \eqref{E-prop}, we can replace $\mathbf{H}_2(\tfrac{N}{2},\tfrac{1}{2})$ by the matrix $\mathbf{H}(\tfrac{N}{2},\tfrac{1}{2})$ appearing in Definition \ref{defn:hyperbola}, up to an error that is sufficiently small to be included in the remainder $R(x,y)$. Therefore, to finish the proof of the proposition, we are left to show that the third order Taylor remainder $R_3(x,y)$ satisfies
\begin{align*}
    \sup_{(x,y)\in D_{2\eta^{2/5}}\left(\tfrac{N}{2},\tfrac{1}{2}\right)}|R_3(x,y)| + \eta^{2/5}|\nabla R_3(x,y)| \leq C_1 \eta^{8/5}/N.
\end{align*}
To do this, it is sufficient to show that
\begin{align*}
    |\nabla^{(\alpha)} v(x,y)| \leq C\eta^{2/5}/N, 
\end{align*}
for all multi-indices $\alpha = (\alpha_1,\alpha_2)$, with $|\alpha| = 3$, and $(x,y)\in D_{2\eta^{2/5}}\left(\tfrac{N}{2},\tfrac{1}{2}\right)$. By Proposition \ref{bounds-prop} \eqref{v1-prop} and \eqref{E-prop}, we can immediately reduce to showing that
\begin{align*}
    |\nabla^{(\alpha)} \left(v_2(x)\sin(2\pi y)\right)| \leq C\eta^{2/5}/N. 
\end{align*}
For $(x,y)$ in $D_{2\eta^{2/5}}\left(\tfrac{N}{2},\tfrac{1}{2}\right)$, we have $|\sin(2\pi y)| \leq 4\pi \eta^{2/5}$. Moreover, by the estimate on $v_2$ and its derivatives in Proposition \ref{bounds-prop} \eqref{v2-prop}, 
\begin{align*}
|v_2(x)| \leq C\eta^{2/5}N^{-1}, \quad |v_2'(x)|\leq CN^{-1}, \quad |v_2'''(x)|\leq CN^{-3}.
\end{align*}
Finally, since $v_2(x)$ satisfies the equation $v_2''(x) = (\mu - 4\pi^2)v_2(x)$, with $|\mu-4\pi^2| \leq CN^{-2}$, its second derivative can be therefore bounded by $C\eta^{2/5}N^{-3}$. In particular, all third derivatives of $v_2(x)\sin(2\pi y)$ can be bounded by $C\eta^{2/5}/N$ as required, completing the proof of Proposition \ref{prop:hyperbola-approx}.
\end{proof1}
\begin{proof1}{Proposition \ref{prop:hyperbola}}
We start by writing the polynomial $P(x,y)$ from Definition \ref{defn:hyperbola} as
\begin{align} \label{eqn:quadratic}
    \alpha p^2 + 2\gamma pq + \beta q^2 + 2ap + 2bp + c.
\end{align}
Here $p = x- \tfrac{N}{2}$, $q = y-\tfrac{1}{2}$. Since
\begin{align*}
    & v(\tfrac{N}{2},\tfrac{1}{2}) = v_1(\tfrac{N}{2}) + E(\tfrac{N}{2},\tfrac{1}{2}), \quad \pa_yv(\tfrac{N}{2},\tfrac{1}{2}) = -2\pi v_2(\tfrac{N}{2}) + \pa_yE(\tfrac{N}{2},\tfrac{1}{2}), \\
    & \hspace{2cm} \pa_xv(\tfrac{N}{2},\tfrac{1}{2}) = v_1'(\tfrac{N}{2}) + \pa_xE(\tfrac{N}{2},\tfrac{1}{2}),
\end{align*}
by Proposition \ref{bounds-prop} \eqref{v1-prop} and \eqref{E-prop}, the coefficients $a$, $b$, and $c$ satisfy
\begin{align} \label{eqn:coefficients1}
    |a| \leq C\eta/N, \quad |b| \leq C\eta/N, \quad C^{-1}\eta N^{-1} \leq |c| \leq C\eta N^{-1}.
\end{align}
Note that to obtain the upper and lower bound on $|c|$ we have used Assumption \ref{assum:phi} together with the estimate from Proposition \ref{bounds-prop} \eqref{v1-prop} that says that $|v_1(\tfrac{N}{2})|$ is comparable to $|v_1(0)|$. The same proposition implies that the coefficients $\alpha$, $\beta$, $\gamma$ satisfy
\begin{align} \label{eqn:coefficients2}
    |\alpha| \leq C\eta/N, \quad |\beta| \leq C\eta/N, \quad C^{-1}N^{-1} \leq |\gamma| \leq CN^{-1}.
\end{align}
The quadratic form in \eqref{eqn:quadratic} will describe a hyperbola provided $\begin{vmatrix} \alpha & \gamma \\ \gamma & \beta \end{vmatrix} = \det \mathbf{H} < 0$. By \eqref{eqn:coefficients2}, we have the following lower bound
\begin{align*}
|\det \mathbf{H}| \geq C^{-1} N^{-2} - C\left(\tfrac{\eta}{N}\right)^2 \geq \tfrac{1}{2}C^{-1} N^{-2}
\end{align*}
for $\eta>0$ sufficiently small, and so the quadratic form is indeed a hyperbola. Letting $(x_c,y_c)$ denote the center of the hyperbola, we may write the hyperbola as
\begin{equation}\label{pre-simplified-hyperbola}
\mu_1^2 v_1\left(\tfrac{N}{2}\right)(x - x_c)^2 + 4\pi v'_2\left(\tfrac{N}{2}\right) (x-x_c)(y-y_c) + \pi^2 v_1\left(\tfrac{N}{2}\right) (y - y_c)^2 =\frac{D}{\det \mathbf{H}},
\end{equation}
where
\begin{align*}
D = \begin{vmatrix}
\alpha & \gamma & a \\
\gamma & \beta & b \\
a & b & c
\end{vmatrix} = -(\alpha b^2 + \beta a^2) + 2ab\gamma + c\det\mathbf{H}.
\end{align*}
Using Proposition \ref{bounds-prop}, and the lower bound on $\det \mathbf{H}$, we have
\begin{align*}
\left| \frac{D}{\det \mathbf{H}} - c\right| = \left| \frac{-(\alpha b^2 + \beta a^2) + 2ab\gamma}{\det \mathbf{H}}\right| &\leq C\eta^2/N.
\end{align*}
Therefore, using the bounds on $|c|$ in \eqref{eqn:coefficients1}, the right hand side of \eqref{pre-simplified-hyperbola} is comparable to $\eta /N$ for $\eta>0$ sufficiently small. This gives the desired bound on $|P_0|$. Since the center $(x_c,y_c)$ is given by
\begin{align*}
    x_c -\tfrac{N}{2} = -\frac{1}{\det\mathbf{H}}\begin{vmatrix}
a & \gamma  \\
b & \beta 
\end{vmatrix} \qquad y_c - \tfrac{1}{2} = - \frac{1}{\det\mathbf{H}} \begin{vmatrix}
\alpha & a  \\
\gamma & b 
\end{vmatrix},
\end{align*}
we also obtain the estimates on $(x_c,y_c)$ in i) in the Proposition.
\\
\\
The principal axes of the hyperbola makes an angle $\varphi$ with the positive $x$-axis where $\varphi$ satisfies
\begin{align} \label{eqn:tangent}
    \tan(2\varphi) = \frac{2\gamma}{\alpha - \beta} = \frac{8\pi v_2'(\tfrac{N}{2})}{(\mu_1^2 - \pi^2)v_1(\tfrac{N}{2})}. 
\end{align}
Using Assumptions \ref{assum:simple} and \ref{assum:phi}, the denominator is comparable to $\eta /N$, while by Proposition \ref{bounds-prop} \eqref{v2-prop}, the numerator is comparable to $N^{-1}$. Therefore,
\begin{align*}
    \left|\cot(2\varphi)\right| = \frac{1}{\left|\tan(2\varphi)\right|} \leq C\eta,
\end{align*}
giving the estimate $\left|2|\varphi| - \tfrac{\pi}{2}\right| \leq C\eta$. Moreover, since $\mu_1^2-\pi^2>0$, we see from \eqref{eqn:tangent}, that the sign of $\tan(2\varphi)$ is equal to the sign of $v_2'(\tfrac{N}{2})/v_1(\tfrac{N}{2})$, which hence proves part iii) of the Proposition.
\\
\\
Finally, rotating the coordinates by this angle $\varphi$, the hyperbola can be written as 
    $\frac{\tilde{x}^2}{\tilde{a}^2} - \frac{\tilde{y}^2}{\tilde{b}^2} = 1.$
Here 
\begin{align*}
    \tilde{a}^2 = \left|\frac{D}{\lambda_1\det\mathbf{H}}\right|, \quad   \tilde{b}^2 = \left|\frac{D}{\lambda_2\det\mathbf{H}}\right|,
\end{align*}
where $\lambda_1$ and $\lambda_2$ are the roots of the quadratic $\lambda^2 -(\alpha+\beta)\lambda + \det\mathbf{H}=0$. By the estimates in \eqref{eqn:coefficients2}, $|\lambda_1|$, $|\lambda_2|$ are comparable to $N^{-1}$, and $D/\det\mathbf{H}$ is comparable to $\eta/N$. In particular, $\tilde{a}^2$ is comparable to $\eta$. 
This ensures that the distance between the vertices of the hyperbola is comparable to $\sqrt{\eta}$, completing the proof of ii) in the Proposition. 
\end{proof1}

\section{Estimates on the Nodal Set Away from the Center} \label{sec:edge}

In this section we will study the nodal set of $v$ outside of a square of side length $C^*\sqrt{\eta}$, centered at $(\tfrac{N}{2},\tfrac{1}{2})$. Fixing $C^*>0$ to be a sufficiently large absolute constant, by the work in the previous section, the nodal set of $v$ enters this square along precisely four curves, one for each side of the square. We will now show that the nodal set outside of this square consists of an extension of these four curves to the boundary of $\Omega$. We will prove that two of the curves are almost horizontal and contained in the strip $\left[\tfrac{1}{2}-\tfrac{1}{2}C^*\sqrt{\eta},\tfrac{1}{2}+\tfrac{1}{2}C^*\sqrt{\eta}\right]$ and the other two curves are almost vertical and contained in the strip $\left[\tfrac{N}{2}-\tfrac{1}{2}C^*\sqrt{\eta},\tfrac{N}{2}+\tfrac{1}{2}C^*\sqrt{\eta}\right]$. The parts of these curves that enter the parts of $\Omega$ with $x<1$ or $x>N-1$ and $|y-\tfrac{1}{2}| \leq \tfrac{1}{4}$, before touching the left and right boundaries of $\Omega$ require some extra analysis, which we will carry out in Section \ref{sec:perpendicular}. The results in this section are valid for all $N>4$ and for $\eta>0$ sufficiently small independent of the size of $N$. 

\begin{Def}
In the current section, we denote $\tilde{\Omega}$ to be the part of $\Omega$ outside of the square centered at $(\tfrac{N}{2},\tfrac{1}{2})$, and excluding these neighborhoods of the left and right boundary sides of $\Omega$. 
\end{Def}

\begin{prop} \label{prop:horizontal}
There exist constants $\eta_0>0$ and $C_1>0$ such that for all $N>4$ and $0<\eta<\eta_0$, the nodal set of $v$ has the following properties: Let $(x_0,y_0)$ be in the nodal set $v^{-1}(0)$, with $(x_0,y_0)\in\tilde{\Omega}$ and $|x_0-\tfrac{N}{2}|>\tfrac{1}{2}C^*\sqrt{\eta}$, $|y_0-\tfrac{1}{2}| <\tfrac{1}{4}$. Then, there exists a function $h(x)$ and an open set $U$ containing $(x_0,y_0)$ such that $v^{-1}(0)\cap U = \{(x,h(x))\}$. Moreover, if $|x_0-\tfrac{N}{2}|>\tfrac{1}{10}$, then for all $(x,h(x))\in U$, 
\begin{align*}
    \left|h(x)-\tfrac{1}{2}\right|  + |h'(x)| \leq \frac{C_1 \eta}{N\left|\sin\left(\tfrac{2\pi}{N}x\right)\right| + 1},
\end{align*}
while if $\tfrac{1}{2}C^*\sqrt{\eta} \leq |x_0-\tfrac{N}{2}|\leq \tfrac{1}{10}$, then for all $x\in U$,
\begin{align*}
    \left|h(x)-\tfrac{1}{2}\right| \leq \frac{C_1 \eta}{\left|x-\tfrac{N}{2}\right|}, \qquad
    |h'(x)|  \leq  \frac{C_1 \eta}{\left|x-\tfrac{N}{2}\right|^2}.
\end{align*}
\end{prop}
\begin{prop} \label{prop:vertical}
There exist constants $\eta_0>0$ and $C_2>0$ such that for all $N>4$ and $0<\eta<\eta_0$, the nodal set of $v$ has the following properties: Let $(x_0,y_0)$ be in the nodal set $v^{-1}(0)$, with $(x_0,y_0)\in\tilde{\Omega}$ and $|y_0-\tfrac{1}{2}|>\tfrac{1}{2}C^*\sqrt{\eta}$. Then, there exists a function $g(y)$ and an open set $V$ containing $(x_0,y_0)$ such that $v^{-1}(0)\cap V = \{(g(y),y)\}$. Moreover, for all $(g(y),y)\in V$, 
\begin{align*}
    \left|g(y)-\tfrac{N}{2}\right| \leq \frac{C_2 \eta}{\left|y-\tfrac{1}{2}\right|}, \qquad
    |g'(y)|  \leq  \frac{C_2 \eta}{\left|y-\tfrac{1}{2}\right|^2}.
\end{align*}
\end{prop}
Combining Propositions \ref{prop:horizontal} and \ref{prop:vertical} with the work in Section \ref{sec:center}, this completes the proof of Theorems  \ref{thm:nodal-gap-estimate} and \ref{thm:regularity}, except for the orthogonality statement in Theorem \ref{thm:regularity} iii).  To prove these propositions, we will use a different decomposition of $v$ than for the center analysis in Section \ref{sec:center}. This is because in this section, the second Fourier mode $v_2(x)\sin(2\pi y)$ will provide the major contribution to the eigenfunction, and we can group the first Fourier mode with the rest of the error term. We therefore write
\begin{align} \label{eqn:B-decom}
v(x,y) = v_2(x)\sin(2\pi y) + B(x,y),
\end{align}
where $B(x,y) = v_1(x)\sin(\pi y) + E(x,y)$. By the properties of $v_1(x)$ and $E(x,y)$ in Proposition \ref{bounds-prop} \eqref{v1-prop} and \eqref{E-prop}, $B(x,y)$ satisfies
\begin{align} \label{eqn:B-bounds1}
   \sup_{(x,y)\in\Omega}|B(x,y)| + |\nabla B(x,y)| \leq C\eta/N
\end{align}
and
\begin{align} \label{eqn:B-bounds2}
   \sup_{ y\in[1/10,9/10]}\left|\nabla^j B(x,y) \right|  +   \displaystyle \sup_{x\in [N/4,3N/4]} \Big| \nabla^j B(x,y) \Big| \leq C\eta/N , \text{ for }  j =2,3.
\end{align}
We will also need to use the approximation on $v_2(x)$ from Proposition \ref{bounds-prop}  \eqref{v2-prop}, 
\begin{align} \label{eqn:horizontal1}
    \left|v_2(x) - \sin\left(\tfrac{2\pi}{N}x\right)\right| \leq C\eta/N.
\end{align}
Combining this with the estimate on $B(x,y)$ in \eqref{eqn:B-bounds1}, this implies that
\begin{align} \label{eqn:horizontal2}
    \left|v(x,y) - \sin\left(\tfrac{2\pi}{N}x\right)\sin(2\pi y )\right| \leq C\eta/N.
\end{align}

\begin{proof1}{Proposition \ref{prop:horizontal}}
We first obtain estimates on the location of the nodal set of $v$. Note that for $y\in[\tfrac{1}{4},\tfrac{3}{4}]$, $|\sin(2\pi y)|\geq 4|y-\tfrac{1}{2}|$. Therefore, for $(x,y)\in\tilde{\Omega}$, with $|y-\tfrac{1}{2}|<\tfrac{1}{4}$, the estimate in \eqref{eqn:horizontal2} gives the lower bound
\begin{align} \label{eqn:horizontal3}
    |v(x,y)| \geq |\sin\left(\tfrac{2\pi}{N}x\right)\sin(2\pi y )| - C\eta/N \geq 4|y-\tfrac{1}{2}|\left|\sin\left(\tfrac{2\pi}{N}x\right)\right| - C\eta/N.
\end{align}
For $|x-\tfrac{N}{2}|>\tfrac{1}{10}$, with $1<x<N-1$, 
\begin{align*}
    \left|\sin\left(\tfrac{2\pi}{N}x\right)\right| \geq cN^{-1},
\end{align*}
while for $\tfrac{1}{2}C^*\sqrt{\eta}\leq |x-\tfrac{N}{2}| \leq \tfrac{1}{10}$,
\begin{align*}
    \left|\sin\left(\tfrac{2\pi}{N}x\right)\right| \geq c|x-\tfrac{N}{2}|/N.
\end{align*}
Therefore, from \eqref{eqn:horizontal3}, given $(x_0,y_0)\in\tilde{\Omega}$ in the nodal set of $v$ with $|y_0-\tfrac{1}{2}| < \tfrac{1}{4}$, if $|x_0-\tfrac{N}{2}|>\tfrac{1}{10}$, then 
\begin{align} \label{eqn:horizontal4}
    |y_0-\tfrac{1}{2}| \leq \frac{C\eta}{N\left|\sin\left(\tfrac{2\pi}{N}x_0\right)\right| + 1},
\end{align}
and if $\tfrac{1}{2}C^*\sqrt{\eta}\leq |x_0-\tfrac{N}{2}| \leq \tfrac{1}{10}$, then
\begin{align} \label{eqn:horizontal5}
    |y_0-\tfrac{1}{2}| \leq \frac{C\eta}{|x_0-\tfrac{N}{2}|}.
\end{align}
Once we prove that the nodal set in this part of $\Omega$ can be written as a graph $y=h(x)$, the above estimates on $y_0$ imply the bounds on $|h(x)-\tfrac{1}{2}|$ given in the statement of the proposition. 

To show this graph property, we obtain a lower bound on $\pa_yv(x_0,y_0)$ for $(x_0,y_0)$ in this part of the nodal set of $v$: Using \eqref{eqn:horizontal1} and the bound on $\pa_yB(x_0,y_0)$ from \eqref{eqn:B-bounds1}, we can bound $\pa_yv(x_0,y_0)$ from below by 
\begin{align*} 
    |\pa_yv(x_0,y_0)| = |2\pi v_2(x_0)\cos(2\pi y_0) + \pa_yB(x_0,y_0)| \geq 2\pi\left|\sin\left(\tfrac{2\pi}{N}x_0\right)\right||\cos(2\pi y_0)| - C\eta/N.
\end{align*}
We have $\left|\sin\left(\tfrac{2\pi}{N}x_0\right)\right|\geq c\sqrt{\eta}/N$ for $\tfrac{1}{2}C^*\sqrt{\eta} < \left|x_0-\tfrac{N}{2}\right| < \tfrac{N}{2}-1$. Moreover, using the bounds on the location of the nodal set in \eqref{eqn:horizontal4} and \eqref{eqn:horizontal5}, $|\cos(2\pi y_0)|$ is certainly bounded below. Therefore, on this part of the nodal set,
\begin{align} \label{eqn:horizontal6}
     |\pa_yv(x_0,y_0)| \geq c\left|\sin(\tfrac{2\pi}{N}x)\right| > 0.
\end{align}
This gives the existence of a neighborhood around each such point in the nodal set of $v$, on which the nodal set is given by $y=h(x)$. The derivative of $h(x)$ satisfies
\begin{align*}
    h'(x) = -\frac{\pa_xv(x,h(x))}{\pa_yv(x,h(x))}.
\end{align*}
Using \eqref{eqn:B-bounds1}, together with $|v_2'(x_0)|\leq CN^{-1}$ from Proposition \ref{bounds-prop} (2), we can bound $\pa_{x}v(x_0,y_0)$ by
\begin{align} \label{eqn:horizontal7}
    |\pa_xv(x_0,y_0)| = |v_2'(x_0)\sin(2\pi y_0) + \pa_xB(x_0,y_0)| \leq C|\sin(2\pi y_0)|/N + C\eta/N.
\end{align}
Dividing \eqref{eqn:horizontal7} by \eqref{eqn:horizontal6}, and using the bounds on the location of the nodal set in \eqref{eqn:horizontal4} and \eqref{eqn:horizontal5} we therefore have
\begin{align*}
    |h'(x)| \leq C \frac{\eta/N}{\left|\sin\left(\frac{2\pi}{N} x\right) \right|} \leq C\frac{\eta}{N\left|\sin\left(\frac{2\pi}{N} x\right) \right|+1} 
\end{align*}
for $|x-\tfrac{N}{2}|>\tfrac{1}{10}$, and
\begin{align*}
    |h'(x)| \leq C \frac{\frac{\eta}{N\left|x-\tfrac{N}{2}\right|}}{\left|x-\tfrac{N}{2} \right|/N}   = C\frac{\eta}{\left|x-\tfrac{N}{2} \right|^2}
\end{align*}
for $\tfrac{1}{2}C^*\sqrt{\eta} \leq |x-\tfrac{N}{2}| \leq \tfrac{1}{10}$. These bounds complete the proof of the proposition.
\end{proof1}
\begin{proof1}{Proposition \ref{prop:vertical}}
To prove this proposition, we work in the part of $\tilde{\Omega}$ with $|y-\tfrac{1}{2}|>\tfrac{1}{2}C^*\sqrt{\eta}$, which we split into three regions given by a) $1<x<N-1$, $|y-\tfrac{1}{2}|<\tfrac{1}{4}$, b) $1<x<N-1$, $|y-\tfrac{1}{2}|\geq\tfrac{1}{4}$, and c) $x<1$ or $x>N-1$, $|y-\tfrac{1}{2}|\geq\tfrac{1}{4}$. We first obtain estimates on the nodal set. \\
\\
In region a), we use the estimate on $v(x,y)$ from \eqref{eqn:horizontal2}, together with $|\sin(2\pi y)| \geq 4|y-\tfrac{1}{2}|$ for $|y-\tfrac{1}{2}|<\tfrac{1}{4}$, to obtain the lower bound
\begin{align*}
    |v(x,y)| \geq 4\left|\sin\left(\tfrac{2\pi}{N}x\right)\right|\left|y-\tfrac{1}{2}\right| - C\eta/N. 
\end{align*}
Since $|y-\tfrac{1}{2}|\geq \tfrac{1}{2}C^*\sqrt{\eta}$ in the part of $\tilde{\Omega}$ under consideration in this proposition, this estimate ensures that $v(x,y)$ cannot equal $0$ for $1<\left|x-\tfrac{N}{2}\right|<\tfrac{N}{2}-1$. Moreover, using
$\left|\sin\left(\tfrac{2\pi}{N}x\right)\right| \geq c\left|x-\tfrac{N}{2}\right|/N$ for $\left|x-\tfrac{N}{2}\right|<1$, this means that $v(x,y)\neq0$ when $|x-\tfrac{N}{2}|\geq C\eta/|y-\tfrac{1}{2}|$.
\\
\\
In region b), we can restrict to $1<x<N-1$, with $y\in[0,\tfrac{1}{4}]$, as the same argument will work for $y\in[\tfrac{3}{4},1]$. From $v(x,0)=B(x,0)\equiv0$, and using the first derivative bounds on $B(x,y)$ in \eqref{eqn:B-bounds1}, we have
\begin{align*}
    |v(x,y)-v_2(x)\sin(2\pi y)| = |B(x,y)| \leq C|y|\eta/N. 
\end{align*}
For $y\in[0,\tfrac{1}{4}]$, note that $\sin(2\pi y) \geq 4|y|$, and so
\begin{align*}
    |v(x,y)| \geq 4|v_2(x)||y|- C|y|\eta/N.
\end{align*}
The estimate from \eqref{eqn:horizontal1} therefore implies that $v(x,y)\neq0$ provided $\left|\sin\left(\tfrac{2\pi}{N}x\right)\right|\geq C_2\eta/N$, for a sufficiently large constant $C_2$. In particular, this shows that this part of the nodal set is contained in a strip around $x=\tfrac{N}{2}$ of size bounded by a multiple of $\eta$, giving the desired estimate on the location of the nodal line in region b).
\\
\\
In region c), we will consider the region $x<1$, $0\leq y \leq \tfrac{1}{4}$, as the other three corners in this region can be handled identically. By the estimates in regions a) and b), we know that no part of the nodal set can enter this region through the line $x=1$ or $y =\tfrac{1}{4}$. Therefore, if the nodal set intersects this region, there must be a nodal domain of $v$ contained in the part of $\Omega$ with $x<1$ and $y<\tfrac{1}{4}$, and the restriction of $v$ to this nodal domain is a first eigenfunction of the nodal domain with eigenvalue $\mu$. However, as this part of $\Omega$ is contained in the strip $0 \leq y \leq \tfrac{1}{4}$, by domain monotonicity its first eigenvalue is bounded below by $16\pi^2$, which is strictly greater than $\mu$. Therefore, the nodal set of $v$ does not intersect region c). 
\\
\\
Therefore, once we have shown that the nodal set in regions a) and b) can be written as a graph $x=g(y)$, we have established the estimate on $|g(y)-\tfrac{N}{2}|$ in the statement of the proposition. To show that this part of the nodal set is a graph $x=g(y)$, we need to show that $\pa_xv(x_0,y_0)\neq0$ whenever $(x_0,y_0)$ is in the nodal set in regions a) and b). Once we have established this, then the derivative of $g(y)$ is given by $g(y)= -\pa_yv(g(y),y)/\pa_xv(g(y),y)$.
\\
\\
In region a), fix $(x_0,y_0)$ in the nodal set of $v$, so that in particular $|x_0-\tfrac{N}{2}|<C\eta/|y_0-\tfrac{1}{2}|<C\eta^{1/2}$. This in particular guarantees from the estimate on $v_2'(x)$ in Proposition \ref{bounds-prop} \eqref{v2-prop} that $|v_2'(x_0)|$ is comparable to $N^{-1}$. Therefore, using the bound on $\pa_xB(x_0,y_0)$ from \eqref{eqn:B-bounds1},
\begin{align*}
    |\pa_xv(x_0,y_0)| = |v_2'(x_0)\sin(2\pi y_0) + \pa_xB(x_0,y_0)|\geq c|y_0-\tfrac{1}{2}|/N- C\eta/N.
\end{align*}
 Since in the domain $\tilde{\Omega}$, we have $|y_0-\tfrac{1}{2}|\geq\tfrac{1}{2}C^*\sqrt{\eta}$, for $\eta>0$ sufficiently small,
\begin{align} \label{eqn:graph1}
    |\pa_xv(x_0,y_0)| \geq \tfrac{1}{2}c|y_0-\tfrac{1}{2}|/N>0,
\end{align}
and this part of the nodal set can be expressed as a graph $x=g(y)$. For an upper bound on $\pa_yv(x_0,y_0)$, we first use \eqref{eqn:B-bounds1}, to obtain the estimate
\begin{align*}
    |\pa_yv(x_0,y_0)| = |2\pi v_2(x_0)\cos(2\pi y_0) + \pa_yB(x_0,y_0)| \leq 2\pi|v_2(x_0)||\cos(2\pi y_0)| + C\eta/N.
\end{align*}
By Proposition \ref{bounds-prop} \eqref{v2-prop}, $|v_2(x_0)| \leq C|x_0-x^*|/N \leq C|x_0-\tfrac{N}{2}|/N + C\eta/N$. Combining this with the location estimate $\left|x_0-\tfrac{N}{2}\right|\leq C\eta/|y_0-\tfrac{1}{2}|$, we have
\begin{align} \label{eqn:graph2}
 |\pa_yv(x_0,y_0)| \leq C\left|x_0-\tfrac{N}{2}\right|/N + C\eta/N \leq C\frac{\eta}{N|y_0-\tfrac{1}{2}|}.
\end{align}
Dividing \eqref{eqn:graph2} by \eqref{eqn:graph1} gives the desired bound on $g'(y)$.
\\
\\
In region b), we again restrict to the part with $y\in[0,\tfrac{1}{4}]$. Since $\sin(2\pi y)$ equals $0$ at $y=0$, we need a modified argument for this part of $\tilde{\Omega}$. Again for $(x_0,y_0)$ in the nodal set of $v$ we have that $|v_2'(x_0)|$ is comparable to $N^{-1}$. Using $v(x,0) = B(x,0) \equiv 0$, by \eqref{eqn:B-bounds2} we therefore have the bound
\begin{align*}
    |\pa_{x}v(x_0,y_0)| = |v_2'(x_0)\sin(2\pi y_0) + \pa_xB(x_0,y_0)| \geq c|y_0|/N - C\eta|y_0|/N.
\end{align*}
Thus, for $\eta>0$ sufficiently small,
\begin{align} \label{eqn:graph3}
    |\pa_xv(x_0,y_0)|\geq \tfrac{1}{2}c|y_0|/N>0,
\end{align}
guaranteeing that this part of the nodal set is a graph. We now seek an upper bound on $\pa_yv(x_0,y_0)$ that contains a factor of $|y_0|$. Since $v(x_0,y_0)=0$, we have $v_2(x_0) = -\frac{B(x_0,y_0)}{\sin(2\pi y_0)}$. This lets us write
\begin{align} \nonumber
    \pa_yv(x_0,y_0) & = 2\pi v_2(x_0)\cos(2\pi y_0) + \pa_yB(x_0,y_0) \\ \label{eqn:graph4}
    & = -2\pi\frac{\cos(2\pi y_0)}{\sin(2\pi y_0)}B(x_0,y_0) + \pa_yB(x_0,y_0).
\end{align}
By the bounds on $\pa_y^2B(x,y)$ from \eqref{eqn:B-bounds2}, and writing 
\begin{align*}
    B(x_0,y_0) & = y_0\pa_yB(x_0,0) + \tfrac{1}{2}y_0^2\pa^2_yB(x_0,c_1), \\ \pa_yB(x_0,y_0) & = \pa_yB(x_0,0) + y_0\pa^2_yB(x_0,c_2), 
\end{align*}
for some $c_1,c_2\in[0,y_0]$, the expression in \eqref{eqn:graph4} can be written as
\begin{align*}
    \pa_yv(x_0,y_0) = -2\pi\frac{\cos(2\pi y_0)}{\sin(2\pi y_0)}y_0 \pa_yB(x_0,0) + \pa_yB(x_0,0) + \text{ Error}.
\end{align*}
Here the Error can be bounded in absolute value by $C|y_0|\eta/N$. Finally, using $\left|2\pi \frac{\cos(2\pi y_0)}{\sin(2\pi y_0)} - \frac{1}{y_0}\right| \leq C|y_0|$, we obtain
\begin{align} \label{eqn:graph5}
      |\pa_yv(x_0,y_0)| \leq C|y_0|\eta/N.
\end{align}
Dividing \eqref{eqn:graph5} by \eqref{eqn:graph3} gives the desired bound on $g'(y)$, and completes the proof of the proposition. 
\end{proof1}

\section{Estimates on the Nodal Set at the Boundary} \label{sec:perpendicular}

To complete the proof of Theorem \ref{thm:regularity}, we need to establish the orthogonality of the nodal set of $v$ at the four points where it meets the boundary of $\Omega$. Again, this argument is valid for $N>4$, and all $\eta>0$ sufficiently small independent of $N$. The argument is most complicated for the left boundary of $\Omega$ as this is the non-flat side, and so we first focus on this case: Let $v(x_0,y_0)=0$, with $x_0<1$, and suppose that $(x_1,y_1) = (\eta\phi_L(y_1),y_1)$ is the closest point to $(x_0,y_0)$ on $\partial\Omega$. We now rotate the domain so that $(x_1, y_1)$ is directly to the left of $(x_0, y_0)$. That is, we introduce new coordinates $(\tilde{x},\tilde{y}) = F^{-1}(x,y)$, where $F$ is the linear isometry obtained by rotating around $(x_0, y_0)$. We then define $\tilde{v}(\tilde{x},\tilde{y})$ to be equal to the eigenfunction $v$ in these rotated coordinates, $\tilde{v} = v\circ F$. Note that by the bounds on $\phi_L$, the angle of rotation is bounded by $C\eta$. Also, if the left boundary of $\Omega$ is given by $\tilde{x} =\eta\tilde{\phi}_{L}(\tilde{y})$, and the point closest to $(x_0,y_0)$ is thus $(\eta\tilde{\phi}_{L}(y_0),y_0)$ in the new coordinates, then $\tilde{\phi}'_{L}(y_0) = 0$. The eigenfunction in these rotated coordinates has the following decomposition. 
\begin{lemma} \label{lem:rotation}
There exist constants $\eta_0>0$ and $C>0$ such that for all $N>4$ and $0<\eta<\eta_0$, the following holds: For $(\tilde{x},\tilde{y})\in\Omega$, with $\tilde{x}<1$, $\tfrac{1}{5}<\tilde{y}<\tfrac{4}{5}$, the eigenfunction $\tilde{v}$ can be written as
\begin{align*}
    \tilde{v}(\tilde{x},\tilde{y}) = \tilde{w}_2(\tilde{x})\sin(2\pi\tilde{y}) + \tilde{B}(\tilde{x},\tilde{y}),
\end{align*}
with $\tilde{w}_2(\eta\tilde{\phi}_{L}(y_0)) = 0$ and
\begin{align*}
    \left|\tilde{w}_2'(\tilde{x})\right|\sim N^{-1}, \quad \left|\tilde{w}_2''(\tilde{x})\right| \leq CN^{-2}, \quad \left|\tilde{w}_2'''(\tilde{x})\right| \leq CN^{-3}.
\end{align*}
The error term $\tilde{B}(\tilde{x},\tilde{y})$ satisfies
\begin{align*}
    \left|\nabla^j \tilde{B}(\tilde{x},\tilde{y})\right| \leq C\eta/N, \text{ with } 0 \leq j \leq 3.
\end{align*}
\end{lemma}
The estimates in Lemma \ref{lem:rotation} and the properties of the rotation thus ensure that
\begin{align} \label{eqn:rotation1}
    \tilde{w}_2(\tilde{x}_l) = 0 \text{ for } \tilde{x}_l = \eta\tilde{\phi}_L(y_0), \qquad  \tilde{\phi}'_L(y_0) = 0.
\end{align}
Also, differentiating the equation $\tilde{v}(\eta\tilde{\phi}_{L}(\tilde{y}),\tilde{y}) = 0$ with respect to $\tilde{y}$, and using \eqref{eqn:rotation1}, the remainder $\tilde{B}(\tilde{x},\tilde{y})$ satisfies
\begin{align} \label{eqn:tildeB}
    \tilde{B}(\tilde{x}_l,y_0) = \pa_{{\tilde{y}}}\tilde{B}(\tilde{x}_l,y_0) = 0.
\end{align}
Using Lemma \ref{lem:rotation}, we can obtain estimates on the nodal line of $v$ up to the left boundary of $\Omega$.
\begin{prop} \label{prop:orthogonal}
Fix $(x_0,y_0)$ in the nodal set of $v$, with $x_0<1$, and apply the rotation as described above. There exists constants $\eta_0>0$ and $C_1>0$ such that for all $N>4$ and $0<\eta<\eta_0$, in the rotated coordinates with $\tilde{x}$ in a neighborhood of $x_0$, the nodal set of $v$ can be written as the graph $\tilde{y} = \tilde{h}(\tilde{x})$ in an open set $U$ of $(x_0,y_0)$, with
\begin{align*}
  |\tilde h (\tilde x) - \tfrac12 | \leq C_1 \eta, \qquad
  |\tilde{h}'(\tilde{x})| \leq C_1\eta|\tilde{x} - \tilde{x}_l|.
\end{align*}
In particular, this part of the nodal set meets the left boundary of $\Omega$ orthogonally.
\end{prop}
\begin{proof1}{Proposition \ref{prop:orthogonal}}
For ease of notation, when proving this proposition we will now drop the tildes. Throughout, we also take $\eta>0$ sufficiently small so that Propositions \ref{prop:horizontal}, \ref{prop:vertical} and Lemma \ref{lem:rotation} hold. 

The pointwise bound on $ | h (x) - \tfrac12 |$ follows from the expression in Lemma \ref{lem:rotation}: Using $w_2(x_l) = 0$ and the bound on $w_2'$ in Lemma \ref{lem:rotation},  we have $|w_2(x_0)|\sim |x_0-x_l|N^{-1}$. Moreover, since $B(x_l,y_0)=0$, we have $|B(x_0,y_0)|\leq C\eta N^{-1}|x_0-x_l|$. Therefore,
\begin{align*}
    |v(x_0,y_0)| \geq |w_2(x_0)||\sin(2\pi y_0)|-|B(x_0,y_0)| \geq cN^{-1}|x_0-x_l||\sin(2\pi y_0)|- C\eta N^{-1}|x_0-x_l|,
\end{align*}
which is strictly positive for $|y_0-\tfrac{1}{2}|>C\eta$ for $C$ sufficiently large.

To bound $|\tilde{h}'(\tilde{x})|$, we first note that from the above estimate on $|y_0-\tfrac{1}{2}|$ that $|\cos(2\pi y_0)|\geq\tfrac{1}{2}$.   We proceed to obtain a lower bound on $|\pa_y v(x_0,y_0)|$.  Using \eqref{eqn:tildeB},
\begin{align*}
    \pa_yB(x_0,y_0) = (x_0-x_l)\pa_{x}\pa_{y}B(\xi,y_0),
\end{align*}
for some $\xi\in(x_l,x_0)$. Therefore,
\begin{align*}
    |\pa_yv(x_0,y_0)| = \left|2\pi w_2(x_0)\cos(2\pi y)+\pa_yB(x_0,y_0)\right| & \geq c N^{-1}|x_0-x_{l}| - \left|(x_0-x_l)\pa_{x}\pa_{y}B(\xi,y_0)\right| \\& \geq cN^{-1}|x_0-x_l| -C\eta N^{-1}|x_0-x_l|.
\end{align*}
Here $c>0$ and $C$ are constants, so for $\eta>0$ sufficiently small we obtain
\begin{align} \label{eqn:orth1}
    |\pa_yv(x_0,y_0)| \geq cN^{-1}|x_0-x_l|.
\end{align}
This proves the existence of the graph function $y=h(x)$. 
To complete the proof of the proposition, we therefore need to obtain the upper bound on $|\pa_xv(x_0,y_0)|$ of
\begin{align*}
    |\pa_xv(x_0,y_0)| \leq C\eta N^{-1}|x_0-x_l|^2.
\end{align*}
We first use the fact that $v(x_0, y_0) = w_2(x_0) \sin(2\pi y_0) + B(x_0,y_0) = 0$, and solve for $\sin(2\pi y_0)$, so that we can write 
\begin{equation} \label{no-sin-partial-v}
\partial_x v(x_0, y_0) = w_2(x_0)\sin(2\pi y_0) + \pa_xB(x_0,y_0) = -B(x_0, y_0) \frac{w_2'(x_0)}{w_2(x_0)} + \partial_x B(x_0, y_0).
\end{equation}
Taylor expanding $w_2(x_0)$ gives
\begin{align*}
    w_2(x_0) = (x_0 - x_l) w_2'(x_l) + \tfrac{1}{2}(x_0 - x_l)^2 w_2''(x_l) + \tfrac{1}{6}(x_0 - x_l)^3 w_2'''(x_1)
\end{align*}
for some $x_1 \in (x_l, x_0)$. Solving this equation for $w_2'(x_l)$, we can write $w_2'(x_0)$ as
\begin{align} \nonumber
   w_2'(x_0) &= w_2'(x_l) + (x_0 - x_l) w_2''(x_l)+\tfrac{1}{2}(x_0-x_l)^2w_2'''(x_2) \\ \label{eqn:orth2}
&=  \frac{w_2(x_0)}{(x_0 - x_l)} + \tfrac{1}{2} (x_0 - x_l) w_2''(x_l) - \tfrac{1}{6}(x_0 - x_l)^2 w'''(x_1) + \tfrac{1}{2}(x_0 - x_l)^2 w''(x_2) 
\end{align}
for some $x_2\in(x_l,x_0)$. Next, we write the Taylor expansions of $B(x_0, y_0)$ and $\partial_x B(x_0, y_0)$ about $x_l$ to different orders, giving
\begin{align} \label{B-expansion}
B(x_0, y_0) & = (x_0 - x_l)\partial_x B(x_l,y_0) + \tfrac{1}{2} (x_0 - x_l)^2 \partial_x^2 B(x_l,y_0) + \tfrac{1}{6} (x_0 - x_l)^3 \partial_x^3 B(x_3,y_0),\\ \label{partial-B-expansion}
\partial_x B(x_0,y_0) &= \partial_x B(x_l,y_0) + (x_0 - x_l)\partial_x^2  B(x_l,y_0) + \tfrac{1}{2} (x_0 - x_l)^2 \partial_x^2 B(x_4,y_0)
\end{align}
for some $x_3, x_4 \in (x_l, x_0)$. Substituting \eqref{eqn:orth2}, \eqref{B-expansion}, \eqref{partial-B-expansion} into \eqref{no-sin-partial-v}, we obtain the expression for $\pa_xv(x_0,y_0)$ of
\begin{align} \nonumber
\partial_x v(x_0, y_0) = &-\partial_x B(x_l,y_0) - \tfrac{1}{2} (x_0 - x_l) \partial_x^2 B(x_l,y_0) - \tfrac{1}{6}(x_0 - x_l)^2 \partial_x^3 B(x_3,y_0)\\ \nonumber
&\, + \frac{B(x_0, y_0)}{w_2(x_0)}\left[-\tfrac{1}{2} (x_0 - x_l) w_2''(x_l) + \tfrac{1}{6}(x_0 - x_l)^2 w_2'''(x_1) - \tfrac{1}{2}(x_0 - x_l)^2 w_2'''(x_2) \right]\\ \nonumber
&\, + \partial_x B(x_l,y_0) + (x_0 - x_l)\partial_x^2  B(x_l,y_0) + \tfrac{1}{2} (x_0 - x_l)^2 \partial_x^2 B(x_4,y_0)\\ \nonumber
= \,&\,\tfrac{1}{2} (x_0 - x_l) \partial_x^2 B(x_l,y_0)   + (x_0 - x_l)^2\left[ \tfrac{1}{2} \partial_x^2 B(x_4,y_0) - \tfrac{1}{6} \partial_x^3 B(x_3,y_0)\right]\\ \label{eqn:orth3}
& +  \frac{B(x_0, y_0)}{w_2(x_0)}\left[-\tfrac{1}{2} (x_0 - x_l) w_2''(x_l) + \tfrac{1}{6}(x_0 - x_l)^2 w_2'''(x_1) - \tfrac{1}{2}(x_0 - x_l)^2 w_2'''(x_2) \right].
\end{align}
Since $B(x_l,y_0) = 0$, Lemma \ref{lem:rotation} ensures that
\begin{align*}
    \left|\frac{B(x_0,y_0)}{w_2(x_0)}\right| \leq \frac{C|x_0-x_l|\eta N^{-1}}{c|x_0-x_l|N^{-1}}.
\end{align*}
Therefore, from \eqref{eqn:orth3} together with the estimates on the derivatives of $B$ from Lemma \ref{lem:rotation}, we have
\begin{align} \label{eqn:orth4}
\left|\pa_{x}v(x_0,y_0)\right| \leq C\eta N^{-1} |x_0-x_l|.
\end{align}
In order to extract a factor of $(x_0-x_l)^2$ from \eqref{eqn:orth3}, we use the eigenfunction equation $\Delta v + \mu v=0$ to obtain a new expression for $\pa_x^2B(x_l,y_0)$. Since $v(x_l,y_0) = 0$, We can write
\begin{align*}
    \partial_x^2 B(x_l, y_0) = \partial_x^2 v(x_l, y_0) - w_2''(x_l) \sin(2\pi y_0) & = -\partial_y^2 v(x_l, y_0) - \mu v(x_l, y_0) - w_2''(x_l) \sin(2\pi y_0)\\
&= -\partial_y^2 v(x_l, y_0) -  w_2''(x_l) \sin(2\pi y_0)\\
&= -\eta\partial_x v(x_l, y_0) \phi''_L(y_0) - w_2''(x_l) \sin(2\pi y_0),
\end{align*}
where the final line comes from differentiating $v(\eta\phi_L(y_0), y_0) \equiv 0$ twice and using $\eta \phi_L(y_0) = x_l$, $\phi_L'(y_0) = 0$. Inserting this into \eqref{eqn:orth3}, and using $\sin(2\pi y_0) = -B(x_0,y_0)/w_2(x_0)$, gives
\begin{align} \nonumber
    \pa_xv(x_0,y_0) = \,&\,\tfrac{1}{2}\eta (x_0 - x_l)\pa_xv(x_l,y_0)\phi_L''(y_0) +\tfrac{1}{2}(x_0-x_l)w_2''(x_l)\frac{B(x_0,y_0)}{w_2(x_0)}    \\ \label{eqn:orth5}
    & + (x_0 - x_l)^2\left[ \tfrac{1}{2} \partial_x^2 B(x_4,y_0) - \tfrac{1}{6} \partial_x^3 B(x_3,y_0)\right]\\ \nonumber
& +  \frac{B(x_0, y_0)}{w_2(x_0)}\left[-\tfrac{1}{2} (x_0 - x_l) w_2''(x_l) + \tfrac{1}{6}(x_0 - x_l)^2 w_2'''(x_1) - \tfrac{1}{2}(x_0 - x_l)^2 w_2'''(x_2) \right]. 
\end{align}
Notice that the terms containing the factor of $w_2''(x_l)$ cancel. Moreover, by \eqref{eqn:orth4},
\begin{align*}
    |\pa_xv(x_l,y_0)| = |\pa_xv(x_0,y_0) + (x_l-x_0)\pa_x^2v(\xi,y_0)| \leq C N^{-1}|x_l-x_0|,
\end{align*}
where $\xi$ is in $(x_l,x_0)$. Therefore, \eqref{eqn:orth5} implies that $|\pa_xv(x_0,y_0)| \leq C\eta N^{-1}(x_l-x_0)^2$. Combining this with \eqref{eqn:orth1} establishes the estimates in the proposition at $(x_0,y_0)$, and the  bounds obtained above also hold at all points on the nodal set sufficiently close to $(x_0,y_0)$. 
\end{proof1}

\begin{Rem} \label{rem:RHS}
Since the right boundary of $\Omega$ is vertical, an identical proof to that of Proposition \ref{prop:orthogonal}, without the need for rotated coordinates, ensures that the nodal set of $v$ can be written as the graph $y=h(x)$, with $|h'(x)| \leq C\eta|N-x|$ for $x>N-1$.
\end{Rem}
To finish the proof of Theorem \ref{thm:regularity}, we are left to prove the orthogonality of the nodal set at the top and bottom boundaries of $\Omega$.
\begin{prop} \label{prop:orthogonal2}
There exist constants $\eta_0>0$ and $C_2>0$ such that for all $N>4$ and $0<\eta<\eta_0$, the nodal set of $v$ with $y<\tfrac{1}{10}$, can be written as the graph $x=g(y)$ with
\begin{align*}
    |g'(y)| \leq C_2\eta|y|.
\end{align*}
In particular, this part of the nodal set meets the bottom boundary of $\Omega$ orthogonally. The analogous statement is true for the part of the nodal set with $y>\tfrac{9}{10}$.
\end{prop}
\begin{proof1}{Proposition \ref{prop:orthogonal2}}
We fix $(x_0,y_0)$ in the nodal set of $v$ with $y_0<\tfrac{1}{10}$, and we will see that an identical argument works for $y_0>\tfrac{9}{10}$. By Proposition \ref{prop:vertical}, $x_0=g(y_0)$ satisfies $\left|x_0-\tfrac{N}{2}\right|\leq C\eta$, and $|g'(y_0)| \leq C\eta$. In order to prove this proposition, we will use the same lower bound on $\pa_xv(x_0,y_0)$ from \eqref{eqn:graph3} of
\begin{align} \label{eqn:graph6}
    |\pa_x v(x_0,y_0)| \geq \tfrac{1}{2}c|y_0|/N.
\end{align}
However, we need to sharpen the estimate on $\pa_yv(x_0,y_0)$ that we obtained from the expression in \eqref{eqn:graph4},
\begin{align} \label{eqn:graph7}
   \pa_yv(x_0,y_0) = -2\pi\frac{\cos(2\pi y_0)}{\sin(2\pi y_0)}B(x_0,y_0) + \pa_yB(x_0,y_0),
\end{align}
so that it contains a factor of $y_0^2$. We now expand $B(x_0,y_0)$ to its Taylor series of order $3$ about $y_0=0$, and $\pa_yB(x_0,y_0)$ to order $2$. Note that from the eigenfunction equation
\begin{align*}
    \pa_y^2B(x_0,0) = \pa_y^2v(x_0,0) = -\pa_x^2v(x_0,0) - \mu v(x_0,0) = 0.
\end{align*}
Therefore,
\begin{align*}
  -2\pi\frac{\cos(2\pi y_0)}{\sin(2\pi y_0)}B(x_0,y_0) & =   -2\pi\frac{\cos(2\pi y_0)}{\sin(2\pi y_0)}\left(y_0 \pa_yB(x_0,y_0) + \tfrac{1}{6}y_0^3\pa_y^3B(x_0,c_1) \right), \\
  \pa_yB(x_0,y_0) & = \pa_yB(x_0,0) + \tfrac{1}{2}y_0^2\pa_y^2B(x_0,c_2),
\end{align*}
for some $c_1,c_2\in[0,y_0]$. Inserting these expressions in \eqref{eqn:graph7}, and using the bounds on the derivatives of $B$ from \eqref{eqn:B-bounds2}, together with $\left|2\pi \frac{\cos(2\pi y_0)}{\sin(2\pi y_0)} - \frac{1}{y_0}\right| \leq C|y_0|$, we obtain
\begin{align*}
    |\pa_yv(x_0,y_0)| \leq Cy_0^2\eta/N.
\end{align*}
Combining this with \eqref{eqn:graph6} completes the proof of the proposition.
\end{proof1}
We are left to prove Lemma \ref{lem:rotation}.
\begin{proof1}{Lemma \ref{lem:rotation}}
 We define the function  $\tilde{w}_2(\tilde{x})$ from $v_2(x)$ and the rotated coordinates $(\tilde{x},\tilde{y})$ by
\begin{align*}
     \tilde{w}_2(\tilde{x})  =  {v}_2(\tilde{x}) - {v}_2(\tilde{x}_l),
\end{align*}
where $\tilde{x}_{l} =\eta\phi_L(y_0)$. Since the left boundary function $x = \eta \phi_L(y)$ has first derivative bounded by $\eta$, the angle of rotation is bounded by $C\eta$. Therefore, by the properties of $v_2$ from Proposition \ref{bounds-prop} \eqref{v2-prop}, we have ${v}_2(\tilde{x}_l)\leq C\eta/N$ and the function $\tilde{w}_2$ satisfies the bounds listed in the statement of the lemma. Moreover, using the earlier decomposition of $v(x,y) = v_2(x)\sin(2\pi y) + B(x,y)$ from \eqref{eqn:B-decom}, we see that the error $\tilde{B}(\tilde{x},\tilde{y})$ satisfies
\begin{align} \nonumber
    \tilde{B}(\tilde{x},\tilde{y}) & = B(x,y) + v_2(x)\sin(2\pi y) - \tilde{w}_2(\tilde{x})\sin(2\pi \tilde{y}) \\ \label{eqn:rotate1}
    & = B(x,y) + v_2(x)\left(\sin(2\pi y) - \sin(2\pi \tilde{y})\right) + (v_2(x)-v_2(\tilde{x})+{v}_2(\tilde{x}_l))\sin(2\pi\tilde{y}).
\end{align}
For $\tilde{x}<1$, and since the angle of rotation is bounded by $C\eta$, we have $|y-\tilde{y}|,|x-\tilde{x}| < C\eta$. Therefore, using $|v_2(t)| + |v_2'(t)| \leq CN^{-1}$ for all $t\leq 1$, which follows from Proposition \ref{bounds-prop} \eqref{v2-prop},  together with $|\tilde{v}_2(\tilde{x}_l)| \leq C\eta/N$, this ensures that
\begin{align*}
    |\tilde{B}(\tilde{x},\tilde{y}) - B(x,y)| \leq C\eta/N.
\end{align*}
By differentiating \eqref{eqn:rotate1} up to $3$ times, we obtain the same inequality for the derivatives of $\tilde{B}(\tilde{x},\tilde{y})$. The lemma therefore follows from the bounds on $B(x,y)$ in \eqref{eqn:B-bounds1} and \eqref{eqn:B-bounds2}.
\end{proof1}

\section{Proof of Propositions \ref{bounds-prop} and \ref{hadamard-prop}} \label{sec:ansatz}

In this section, we prove Propositions \ref{bounds-prop} and  \ref{hadamard-prop}, using the properties of the domain $\Omega$ and the eigenfunction decomposition. Recall from \eqref{ansatz} and \eqref{mode-inner-product} that for $x\in[0,N]$ we write the eigenfunction $v$ as
\begin{align*}
v(x,y) = \sum_{k \geq 1} v_k(x) \sin(k\pi y) =v_1(x)\sin(\pi y) + v_2(x)\sin(2\pi y) + E(x,y)
\end{align*}
where
\begin{align*}
v_k(x) = 2\int_0^1 v(x,y) \sin(k\pi y).
\end{align*}
To prove Proposition \ref{bounds-prop}, we will write down the ODEs that the Fourier modes $v_k(x)$ satisfy and solve these equations in terms of their values at $x=0$. A similar strategy has been used to prove Proposition 2.1 in \cite{BCM20} which established some analogous properties for a Fourier decomposition of the second eigenfunction. In order to prove Proposition \ref{bounds-prop}, we will need to bound the eigenvalue $\mu$ and the Fourier modes $v_k(0)$. The bound on $\mu$ will follow from the domain monotonicity of Dirichlet eigenvalues in a straightforward manner. The required bound on $v_k(0)$ is $|v_k(0)| \leq C\eta/N$, which from \eqref{mode-inner-product}, will follow from a pointwise estimate on $v(0,y)$ of size $\eta/N$. For the first or second eigenfunction of $\Omega$, such a bound follows from using a comparison function argument (see, for example, Lemma 4.2 in \cite{BCM20}). However, because the nodal set of the eigenfunction $v$ intersects the left boundary, such an argument does not work in this case. Instead, we will use an argument based on a Hadamard variation calculation. A similar calculation has been used in \cite{grieser2009asymptotics} and \cite{BCM20}, but in order to prove $|v_k(0)|\leq C\eta/N$, here we need a sharper dependence on the error in terms of $N$ (see Lemma \ref{lem:vk} below). This proposition will also isolate the main contribution to $v_k(0)$ and prove Proposition \ref{hadamard-prop}. Throughout this section, we fix a value of $N$, with $N>4$, and $\eta$ will be smaller than an absolute constant $\eta_0>0$, which will be chosen in the course of the proofs below.

We start with the bound on the eigenvalue $\mu$.
\begin{lemma} \label{lem:eigenvalue}
The eigenvalue $\mu$, corresponding to the eigenfunction $v$, satisfies
\begin{align} \label{eigenvalue-bound}
4\pi^2\left({\tfrac{1}{(N+\eta)^2} + 1}\right)\leq \mu \leq 4\pi^2\left({\tfrac{1}{N^2} + 1}\right) .
\end{align}
\end{lemma}
\begin{proof}
We observe that
\begin{align*}
[0,N] \times [0,1] \subset \Omega \subset [-\eta,N] \times [0,1]
\end{align*}
and apply domain monotonicity for Dirichlet eigenvalues to obtain the claim.
\end{proof}
From $(\Delta + \mu) v = 0$ and $v(x,y) = \sum_k v_k(x) \sin(k \pi y)$, we see that each $v_k(x)$ obeys
\begin{align}\label{ODE-equations}
v_k'' + (\mu -\pi^2 k^2) v_k = 0.
\end{align}
Moreover, since the right hand side of $\Omega$ is flat, $v_k(N) = 0$ for all $k\geq0$. Therefore, setting $\mu_k^2 = \pi^2k^2-\mu>0$ for $k\geq3$, and $\mu_1^2 = \mu-\pi^2>0$, $\mu_2^2 = \mu-4\pi^2>0$, we obtain the following expressions for $v_k(x)$:
\begin{align} \label{eqn:vk}
    v_k(x) & = \frac{1}{\sinh(\mu_k N)} v_k(0) \sinh\left(\mu_k(N-x)\right), \text{ for }k\geq3, \\ \label{eqn:v1}
    v_1(x) & = v_1(0) \cos (\mu_1 x) + A_1 \sin(\mu_1 x),
\text{ where } A_1 = \frac{ - v_1(0) \cos(\mu_1 N)}{\sin(\mu_1 N)}, \\ \label{eqn:v2}
    v_2(x) & = v_2(0) \cos (\mu_2 x) + A_2 \sin(\mu_2 x),
\text{ where } A_2 = \frac{- v_2(0) \cos(\mu_2 N)}{\sin(\mu_2 N)}.
\end{align}
This means that $|v_k(x)|\leq |v_k(0)|$ for $k\geq3$, and has exponential decay in $N$ away from $x=0$ and $x=N$. That is, for $k\geq3$,
\begin{align} \label{eqn:exponential}
    |v_k(x)| \leq C|v_k(0)| e^{-ck\min\{x,N-x\}},
\end{align}
for a constant $c>0$. By Lemma \ref{lem:eigenvalue}, $\mu_1$, $\mu_2$ satisfy
\begin{align} \label{eqn:mu1-bound}
\frac{4\pi^2}{(N + 2\eta)^2} + 3\pi^2 \leq  \, & \mu_1^2 \leq \frac{4\pi^2}{N^2} + 3\pi^2, \\ \label{eqn:mu2-bound}
\frac{4\pi^2}{(N + 2\eta)^2} \leq \, & \mu_2^2 \leq \frac{4\pi^2}{N^2}.
\end{align}
Therefore, for all $\eta>0$ sufficiently small, by \eqref{eqn:mu1-bound} together with the non-resonant assumption on $N$, Assumption \ref{assum:simple}, this ensures that $\mu_1N$ is bounded away from a multiple of $\pi$. From \eqref{eqn:v1}, this guarantees that $|v_1(x)|\leq C|v_1(0)|$. As a consequence of this, to prove Proposition \ref{bounds-prop}, we need to bound the coefficients $|v_k(0)|$. Since the boundary of $\Omega$ is $C^2-$smooth, except for four points where the side curves meet at convex angles, the gradient $|\nabla v(x,y)|$ is bounded. The definition of $\Omega$, then ensures that $|v(0,y)|\leq C\eta$, and so $|v_k(0)|\leq C\eta$ for all $k\geq1$. However, in order to prove Proposition \ref{bounds-prop}, we will require the stronger estimate, $|v_k(0)|\leq C\eta/N$.
\begin{lemma} \label{lem:v2}
There exist constants $\eta_0>0$ and $C>0$ such that for all $N>4$ and $0<\eta<\eta_0$, the bound $|v_2(0)| \leq C\eta/N$ holds.
\end{lemma}
\begin{proof}
Since $|v_k(0)|\leq C\eta$ for all $k\geq1$, by the expressions in \eqref{eqn:vk}, \eqref{eqn:v1}, and \eqref{eqn:v2}, we have
\begin{align*}
    \left|v(x,y) - A_2\sin(\mu_2 x)\right| \leq C\eta
\end{align*}
for all $(x,y)\in\Omega$ with $x\in[1,N-1]$. Using $\max_{\Omega}|v| = 1$, this ensures that $||A_2|-1|\leq C\eta$. Moreover, from  \eqref{eqn:mu2-bound}, $\mu_2$ satisfies
\begin{align} \label{eqn:mu2-estimate}
    \left|\mu_2N- 2\pi\right| \leq C\eta/N.
\end{align}
Therefore, for $\eta>0$ sufficiently small $\cos(\mu_2N)$ is bounded from below by a positive constant, $\sin(\mu_2N)$ is bounded from above by $C\eta/N$, and so from the expression for $A_2$ in \eqref{eqn:v2}, we must have $|v_2(0)|\leq C\eta/N$.
\end{proof}

In order to use \eqref{eqn:vk}, \eqref{eqn:v1}, and \eqref{eqn:v2} to prove Proposition \ref{bounds-prop}, we need to prove the sharper estimate $|v_k(0)|\leq C\eta/N$ for all $k\geq 1$.
\begin{lemma} \label{lem:vk}
There exist constants $\eta_0>0$ and $C>0$ such that for all $N>4$ and $0<\eta<\eta_0$,
\begin{align*}
    \sum_{k\neq2}v_k(0)^2 \leq C\eta^2/N^2.
\end{align*}
\end{lemma}
\begin{proof}
We will prove this lemma, using a Hadamard variation argument, using a variant of a calculation in \cite{grieser2009asymptotics} and \cite{BCM20}. For $0\leq x \leq N$, we define $B(x,y) = \sum_{k\neq 2}v_k(x)\sin(k\pi y)$, and so
\begin{align*}
    B^2:= \int_{0}^{1}B(0,y)^2\,dy = \sum_{k\neq2}v_k(0)^2\int_{0}^{1}\sin^2(k\pi y)\,dy = \tfrac{1}{2}\sum_{k\neq2}v_k(0)^2.
\end{align*}
Therefore, to prove the claim, it is sufficient to show that $B\leq C\eta/N$. We will first show this under the assumption that
\begin{align} \label{eqn:B-pointwise}
    |B(x,y)| \leq C(B+\eta/N)
\end{align}
holds for all $(x,y)\in\Omega$, with $x\leq 1$. An integration by parts calculation gives an expression for $v_k(0)$ of
\begin{align} \nonumber
    v_k(0) & = 2\int_{0}^{1}v(0,y)\sin(k\pi y)\,dy \\ \label{eqn:vk-1}
    & = -2\int_{\pa\Omega_0}\frac{\pa v}{\pa \nu}x\sin(k\pi y)d\sigma + 2(\mu-k^2\pi^2)\int_{\Omega_0}v(x,y)x\sin(k\pi y)\,dx\,dy.
\end{align}
Here $\Omega_0$ is the portion of $\Omega$ with $x\leq0$, and $d\sigma$ is the measure on $\pa\Omega_0$. Since $\Omega_0$ has area bounded by a multiple of $\eta$, the second integral in \eqref{eqn:vk-1} can be bounded by  $Ck^2\eta^2(B+\eta/N)$. For the first integral in \eqref{eqn:vk-1} we write
\begin{align} \label{eqn:vk-2}
   -2\int_{\pa\Omega_0}\frac{\pa v}{\pa \nu}x\sin(k\pi y)d\sigma = -2\int_{\pa\Omega_0}\frac{\pa (v_2(x)\sin(2\pi y))}{\pa \nu}x\sin(k\pi y)d\sigma -2\int_{\pa\Omega_0}\frac{\pa B}{\pa \nu}x\sin(k\pi y)d\sigma. 
\end{align}
Using elliptic estimates, together with the estimate in \eqref{eqn:B-pointwise}, the $L^2$-norm of $\frac{\pa B}{\pa \nu}$ on $y = \eta\phi_L(y)$ is bounded by $C(B+\eta/N)$. This means that the second integral on the right hand side of \eqref{eqn:vk-2} can be bounded by $C\eta(B+\eta/N)$. The unit normal on $x = \eta\phi_L(y)$ is given by
\begin{align*}
    \nu = (1+\eta^2\phi_L'(y)^2)^{-1/2}\langle -1,-\eta\phi_L'(y)\rangle,
\end{align*}
and since $|v_2'(x)|\leq CN^{-1}$, $|\phi_L(y)|+|\phi_L'(y)| \leq 1$, the first integral on the right hand side of \eqref{eqn:vk-2} can be written as
\begin{align} \label{eqn:vk-3}
   2\int_{0}^{1}v_2'(\eta\phi_L(y))\sin(2\pi y)\eta\phi_L(y)\sin(k\pi y) \, dy   +2\int_{0}^{1}2\pi\eta\phi_L'(y)(v_2(\eta\phi_L(y))\cos(2\pi y))\eta\phi_L(y)\sin(k\pi y) \, dy,
\end{align}
up to an error of size at most $C\eta^3/N$. As $|v_2(0)| \leq C\eta/N$, and $|v_2'(x)| \leq CN^{-1}$, we have $|v_2(\eta \phi_L(y))| \leq C\eta/N$, and so the second integral in \eqref{eqn:vk-3} can also be bounded by $C\eta^3/N$. Putting everything together,
\begin{align} \label{eqn:Hadamard}
v_k(0) =    2\int_{0}^{1}v_2'(\eta\phi_L(y))\sin(2\pi y)\eta\phi_L(y)\sin(k\pi y) \, dy  + \text{ Error},
\end{align}
where the Error can be bounded by
\begin{align*}
    C\left( \eta (B+\eta/N) + k^2\eta^2(B+\eta/N) \right).
\end{align*}
Integrating by parts in the integral in \eqref{eqn:Hadamard}, using $\sin(k\pi y) = -\tfrac{1}{k\pi}\tfrac{d}{dy}\left[\cos(k\pi y)\right]$, it can be bounded by $Ck^{-1}\eta/N$. 
 Putting everything together gives the estimate on $v_k(0)$ of
\begin{align} \label{eqn:vk-4}
|v_k(0)| \leq C\left(k^{-1}\eta/N +  \eta (B+\eta/N) + k^2\eta^2(B+\eta/N) \right).
\end{align}
Alternatively, we can write
\begin{align*}
    v_k(0) = 2\int_{0}^{1}v(0,y)\sin(k\pi y)\,dy & = -\frac{2}{k\pi}\int_{0}^{1}\pa_yv(0,y)\cos(k\pi y)\,dy \\
    & = -\frac{4}{k}\int_{0}^{1}v_2(0)\cos(2\pi y)\cos(k\pi y)\,dy -\frac{2}{k\pi}\int_{0}^{1}\pa_yB(0,y)\cos(k\pi y)\,dy.
\end{align*}
Again using a gradient estimate on $B(x,y)$, we therefore also have the bound 
\begin{align} \label{eqn:vk-5}
    |v_k(0)| \leq Ck^{-1}(\eta/N+B).
\end{align}
This implies that
\begin{align*}
    B^2 = \tfrac{1}{2}\sum_{k\neq2}v_k(0)^2 & =  \tfrac{1}{2}\sum_{k\neq2, k\leq \eta^{-1/2}}v_k(0)^2 + \tfrac{1}{2}\sum_{k> \eta^{-1/2}}v_k(0)^2 \\
    &\leq C\left(\eta^2/N^2 + \eta^{-1/2}\eta^2(B+\eta/N)^2 + \eta^{-5/2}\eta^4(B+\eta/N)^2 \right) + C\left(\eta^{1/2}(\eta/N + B)^2 \right),
\end{align*}
where we have used \eqref{eqn:vk-4} to estimate the first sum and \eqref{eqn:vk-5} to estimate the second sum. In particular, for $\eta>0$ sufficiently small (independently of $N$), we can rearrange this estimate to obtain the required estimate on $B$ of $B^2\leq C\eta^2/N^2$.

We are left to prove the estimate \eqref{eqn:B-pointwise}. The function $B(x,y)$ satisfies the eigenfunction equation
\begin{align*}
    0 = (\Delta + \mu)v(x,y) = (\Delta + \mu)(v_2(x)\sin(2\pi y)) = (\Delta + \mu)B(x,y).
\end{align*}
Moreover, by the ODE solutions for $v_k(x)$ from \eqref{eqn:vk} and \eqref{eqn:v1}, we certainly have
\begin{align*}
    |B(x,y)| \leq C\eta/N
\end{align*}
for all $(x,y)\in\Omega$, with $x\geq\tfrac{1}{10}$. Also, $B(x,y) = 0$ when $y = 0,1$, and  on the left boundary of $\Omega$, with $x = \eta\phi_L(y)$, we have
\begin{align*}
    |B(\eta\phi_L(y),y)| = |v_2(\eta\phi_L(y))\sin(2\pi y)| \leq |v_2(\eta\phi_L(y))| \leq |v_2(0)| + |A_2||\sin(\mu_2\eta\phi_L(y))| \leq C\eta/N,
\end{align*}
since $|\phi_L(y)|\leq 1$. Therefore, writing $\Omega_{1/10}$ for the portion of $\Omega$ with $x\leq\tfrac{1}{10}$, we have $(\Delta + \mu)B = 0$ in $\Omega_{1/10}$, and $|B(x,y)| \leq C\eta/N$ on $\pa\Omega_{1/10}$.
\\
\\
Define a comparison function $R(x,y)$ by
\begin{align*}
    R(x,y) = 100C\tfrac{\eta}{N}\cos\left(\tfrac{\pi}{4}(y-\tfrac{1}{2})\right)\cos(4\pi x).
\end{align*}
Then, $\Delta R = -(\tfrac{1}{16} + 16)\pi^2R$, so that $(\Delta + \mu)R<0$ in $\Omega_{1/10}$. Also $R>0$ in $\Omega_{1/10}$, and $R > C\eta/N\geq|B|$ on $\pa\Omega_{1/10}$. By the generalized maximum principle,
\begin{align*}
    |B(x,y)| \leq R(x,y) \leq 100C\eta/N \text{ in } \Omega_{1/10},
\end{align*}
and this completes the proof of the claim.
\end{proof}
We can now use the estimate $|v_k(0)| \leq C\eta/N$ together with the expressions for $v_k(x)$ from \eqref{eqn:vk}, \eqref{eqn:v1}, and \eqref{eqn:v2} to prove all of the properties listed in Proposition \ref{bounds-prop}.
\begin{proof1}{Proposition \ref{bounds-prop}}
By Lemma \ref{lem:vk}, $v_1(0)$ satisfies $|v_1(0)| \leq C\eta/N$. Moreover, Assumption \ref{assum:simple} ensures that the coefficient $A_1$ from \eqref{eqn:v1} satisfies $|A_1| \leq C|v_1(0)|$, for a constant $C$ depending only on the distance of $N$ from the values $N_k$ appearing in the assumption. Therefore, from the expression for $v_1(x)$ in \eqref{eqn:vk}, it and its derivatives are all bounded by $C\eta/N$. Evaluating the expression for $v_1(x)$ from \eqref{eqn:v1} at $x = \tfrac{N}{2}$ gives
\begin{align*}
v_1(\tfrac{N}{2}) = v_1(0)\left(\cos(\mu_1 \tfrac{N}{2})-\cot(\mu_1N)\sin(\mu_1\tfrac{N}{2})\right) = \frac{v_1(0)}{2\cos(\mu_1\tfrac{N}{2})}.
\end{align*}
Applying Assumption \ref{assum:simple} again thus ensures the existence of a constant $C$ such that
\begin{align*}
    C^{-1}|v_1(0)| \leq v_1(\tfrac{N}{2}) \leq C|v_1(0)|
\end{align*}
and this establishes the estimates in \eqref{v1-prop} in the proposition. 
\\
\\
From Lemma \ref{lem:vk}, together with the exponential decay of $v_k(x)$ for $k\geq3$ from \eqref{eqn:exponential}, we have
\begin{align*}
    \left|v(x)-v_2(x)\sin(2\pi y)\right| \leq C\eta/N
\end{align*}
for all $x\in[1,N-1]$. In particular, since $v$ has $L^{\infty}$-norm equal to $1$, and is normalised to be positive at $(\tfrac{N}{4},\tfrac{1}{4})$, for $\eta>0$ sufficiently small, $v_2(x)$ must be positive at $x=\tfrac{N}{4}$, with a maximum of at least $1-C\eta/N$. The expression for $v_2(x)$ from \eqref{eqn:v2} therefore implies that 
\begin{align*}
    \left|A_2 - 1\right| \leq C\eta/N. 
\end{align*}
The estimates on $v_2(x)$ and its derivatives therefore follow from the expression for $v_2(x)$ in \eqref{eqn:v2} and the estimates on $\mu_2$ in \eqref{eqn:mu2-bound}. The bound
\begin{align*}
   \left|v_2'(x) - \tfrac{2\pi}{N}\cos(\tfrac{2\pi}{N}x)\right| \leq C\eta N^{-2}
\end{align*}
in particular ensures that $|v_2'(x)| \geq C^{-1}N^{-1}$ on $[\tfrac{2N}{5}, \tfrac{3N}{5}]$. Therefore, the function $v_2(x)$ can have at most one zero on this interval. Further, $|v_2(x) - \sin(\tfrac{2\pi}{N} x)| \leq C \frac{\eta}{N}$ and, $\sin(\tfrac{2\pi}{N} x)$ vanishes at $x = \tfrac{N}{2}$. Since $\sin(\tfrac{2\pi}{N}x)$ has derivative bounded below by $C^{-1}N^{-1}$, this ensures that $v_2(x^*) = 0$ for some $x^* \in [\frac{N}{2} - C\eta, \frac{N}{2} + C\eta]$. This completes the proof of the estimates in \eqref{v2-prop} in the proposition.
\\
\\
Finally, Lemma \ref{lem:vk} together with \eqref{eqn:exponential}, implies that $E(x,y)$ is bounded in $L^2(\Omega)$ by $C\eta/N$, and that $|\nabla^j E(x,y)| \leq C\eta e^{-cN}$ for $x\in[\tfrac{1}{4}N,\tfrac{3}{4}N]$, $0\leq j \leq 3$. Since $E$ also satisfies the eigenfunction equation $\Delta E=-\mu E$, the $L^2(\Omega)$ bound on $E(x,y)$ ensures that 
\begin{align*}
    |E(x,y)| + |\nabla E(x,y)| \leq C\eta/N.
\end{align*}
in $\Omega$. For $\tfrac{1}{10} \leq y \leq \tfrac{9}{10}$, we are away from the corners of the domain, and so  as the left boundary $x=\eta\phi_L(y)$ is sufficiently smooth, we can pointwise bound the second and third derivatives of $E(x,y)$ by $C\eta/N$. This gives all the estimates on $E(x,y)$ from \eqref{E-prop} in the proposition, and completes the proof. 
\end{proof1}

We can also use the calculation in Lemma \ref{lem:vk} to prove Proposition \ref{hadamard-prop}. 

\begin{proof1}{Proposition \ref{hadamard-prop}}
By Proposition \ref{bounds-prop} \eqref{v2-prop},
\begin{align*}
    \left|v_2'(x) -\tfrac{2\pi}{N}\cos\left(\tfrac{2\pi}{N}x\right) \right| \leq C\eta N^{-2},
\end{align*}
and so in particular, $\left|v_2'(\eta\phi_L(y))-\tfrac{2\pi}{N}\right| \leq C\eta N^{-2}$. Using this estimate on $v_2'(\eta\phi_L(y))$ in \eqref{eqn:Hadamard}, we see that $v_k(0)$ is equal to a main contribution of
\begin{align*}
    \frac{4\pi\eta}{N}\int_{0}^{1} \phi_L(y)\sin(2\pi y)\sin(k\pi y)\,dy,
\end{align*}
up to an error of size bounded by
\begin{align*}
        C\left( \eta^2/N^3 + \eta (B+\eta/N) + k^2\eta^2(B+\eta/N) \right) \leq C\left(\eta^2/N + k^2\eta^3/N\right), 
\end{align*}
completing the proof of the proposition.
\end{proof1}

\section{Numerical Experiments and Discussion} \label{sec:numerics}

In this section we detail numerical experiments that support our findings and investigate the roles of Assumptions \ref{assum:simple} and \ref{assum:phi}. We consider a few different classes of boundary functions, namely sinusoids as well as some irregular boundary functions including hats and (non-smooth) bumps. To this last point, we relax the assumption that $\phi_L\in C^5([0,1]; \mathbb{R})$ and require it to only be (Lipschitz) continuous. We also loosen the perturbation bound and let $\phi_L(y) \in [-1, 1]$ rather than $[-1, 0]$. Consequently, we also allow for $\phi_L(0), \phi_L(1) \neq 0$.

We begin with a brief description of the numerical methods, and then explore three of our findings, namely that symmetry in the boundary perturbation can close the nodal set, that the nodal set switches orientation about the described resonant $N$, and that the nodal set gap distance is independent of $N$. The numerics support these findings.

\subsection{Description of Numerical Methods}
The analysis was performed using MATLAB, which provides tools for solving PDEs on user-defined domains. Domains were specified using \texttt{polyshape} in MATLAB, which accepts coordinate vectors as an argument and returns a polygon-geometry object. To encode the left boundary perturbation, the interval $[0,1]$ was discretized at steps of size $10^{-3}$, and the boundary functions were evaluated on this mesh. The resulting coordinate pairs were passed to \texttt{polyshape}, along with the other points $(N, 0)$ and $(N, 1)$ on the right side of the domain.

The sinusoidal boundary functions were chosen to be $\phi_L(y) = \cos(6 \pi y)$ and $\phi_L(y) =  \sin(6\pi y)$, with the former being even and the latter odd about $y = \tfrac{1}{2}$. The hat functions were specified using \texttt{triangularPulse}, and the bump functions were specified as crests of a sine wave. Both of these boundary types had finitely many (usually two) hats/bumps, placed at varying $y$.

Once the perturbed rectangular domain was specified, the Dirichlet problem for the Helmholtz equation was solved numerically. MATLAB offers a friendly user-interface for solving boundary value problems in its \texttt{PDE Toolbox}, but also offers a command-line interface for using it. Using this command-line interface, we first generate a finite-element mesh for the domain using \texttt{generateMesh} and solve the Helmholtz problem using \texttt{solvepdeeig}. The mesh resolution can be tuned via an argument to \texttt{generateMesh}, which is useful for ensuring a good rendering of the nodal set for larger domains.

For more details on the implementation, visit \href{https://github.com/marichig/nodal-set-openings}{https://github.com/marichig/nodal-set-openings}, which includes the code used for the following analyses as well as references to the relevant MATLAB documentation.

\subsection{Role of Assumption \ref{assum:phi} and Effect of Symmetric Perturbation on Nodal Set Gap} 
From Proposition \ref{hadamard-prop} and Assumption \ref{assum:phi}, we observe that choices of $\phi_L$ that are symmetric about $y = \tfrac{1}{2}$ would force $v_1(0)$, and consequently $v_1(x)$, to be small. In these cases, any perturbation to the nodal set would be due to the exponentially small error term in (\ref{ansatz}). In fact, the eigenfunction itself will be symmetric about $y = \tfrac{1}{2},$ and so the crossing is genuine, as shown in the plots in Figure \ref{fig:symmetry} (see further below).

Numerics also verify that for asymmetric choices of $\phi_L$ such that Assumption \ref{assum:phi} does not hold, the nodal set opening is due to higher-order error modes and is hence smaller. Figure \ref{fig:non-zero-bdry-integral} demonstrates this with the choice of $\phi_L(y) = \cos(5\pi y).$

\begin{figure}[ht] 
     \begin{subfigure}{0.4\textwidth}
         \centering
         \includegraphics[width=60mm,trim={1.5cm 1.5cm 1.5cm 1.5cm},clip]{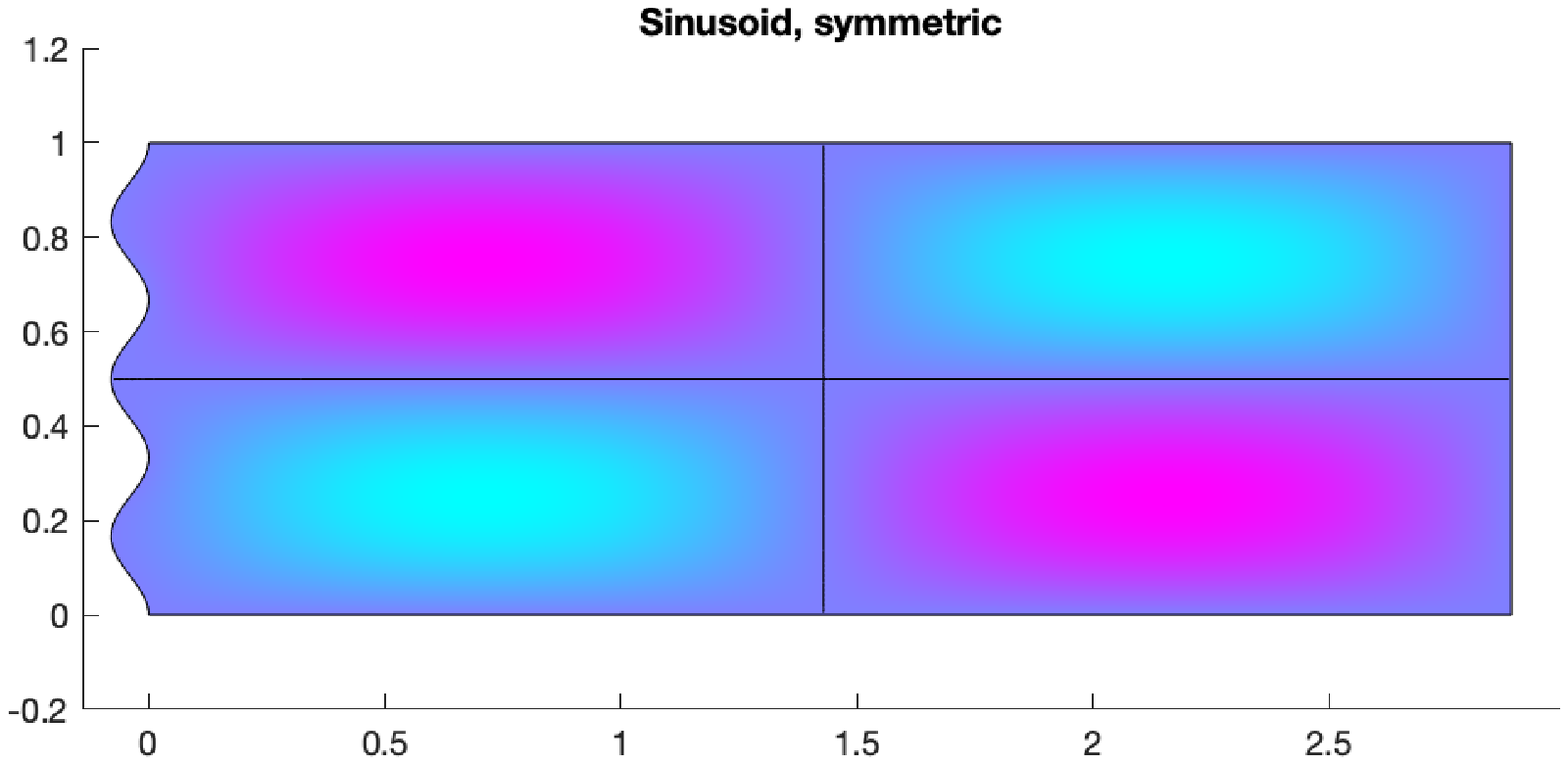}
         \caption{}
     \end{subfigure} \quad
       \begin{subfigure}{0.4\textwidth}
         \centering
         \includegraphics[width=60mm,trim={1.5cm 1.5cm 1.5cm 1.5cm},clip]{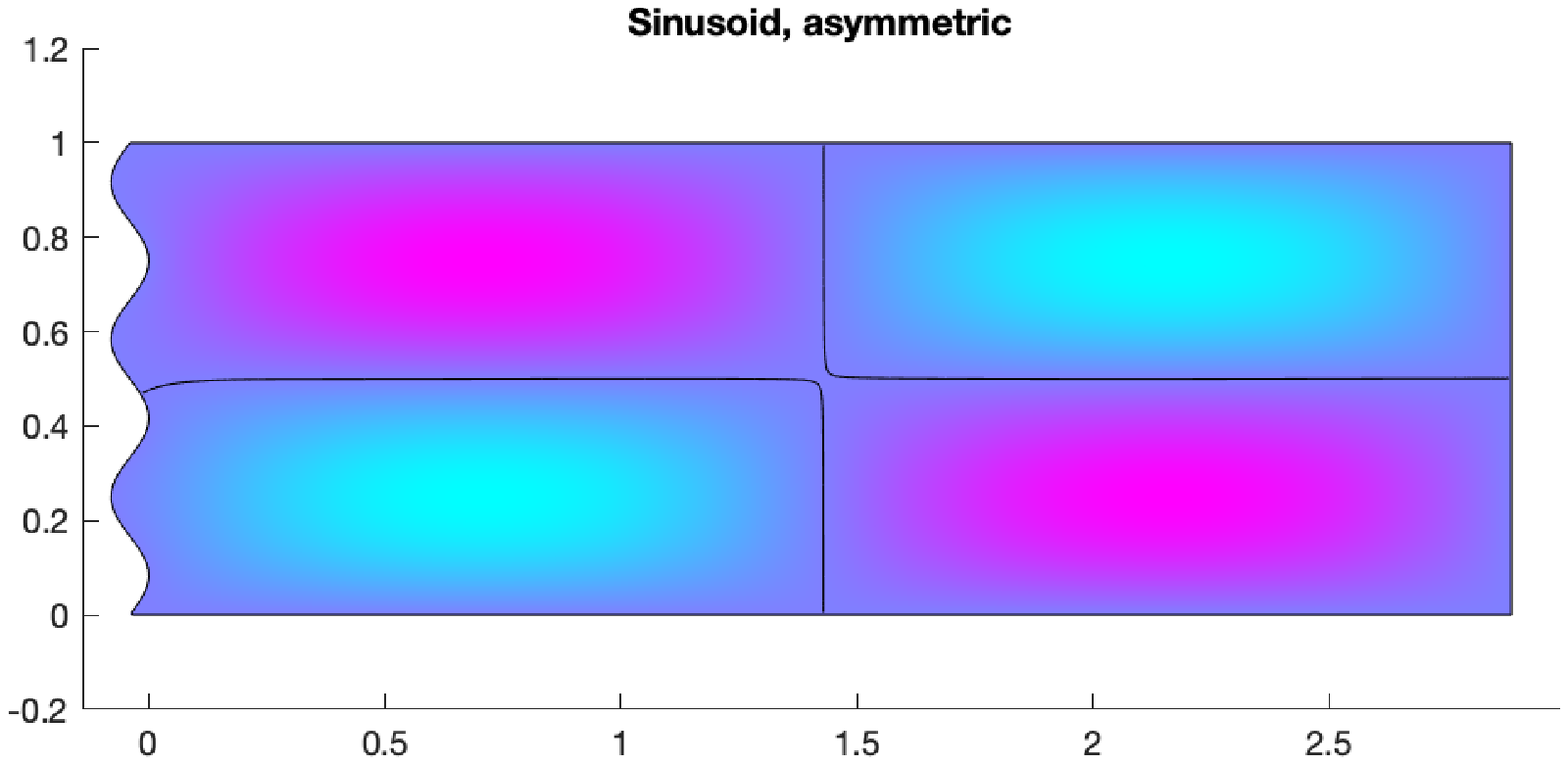}
         \caption{}
     \end{subfigure} \\
     \begin{subfigure}{0.4\textwidth}
         \centering
          \includegraphics[width=60mm,trim={1.5cm 1.5cm 1.5cm 1.5cm},clip]{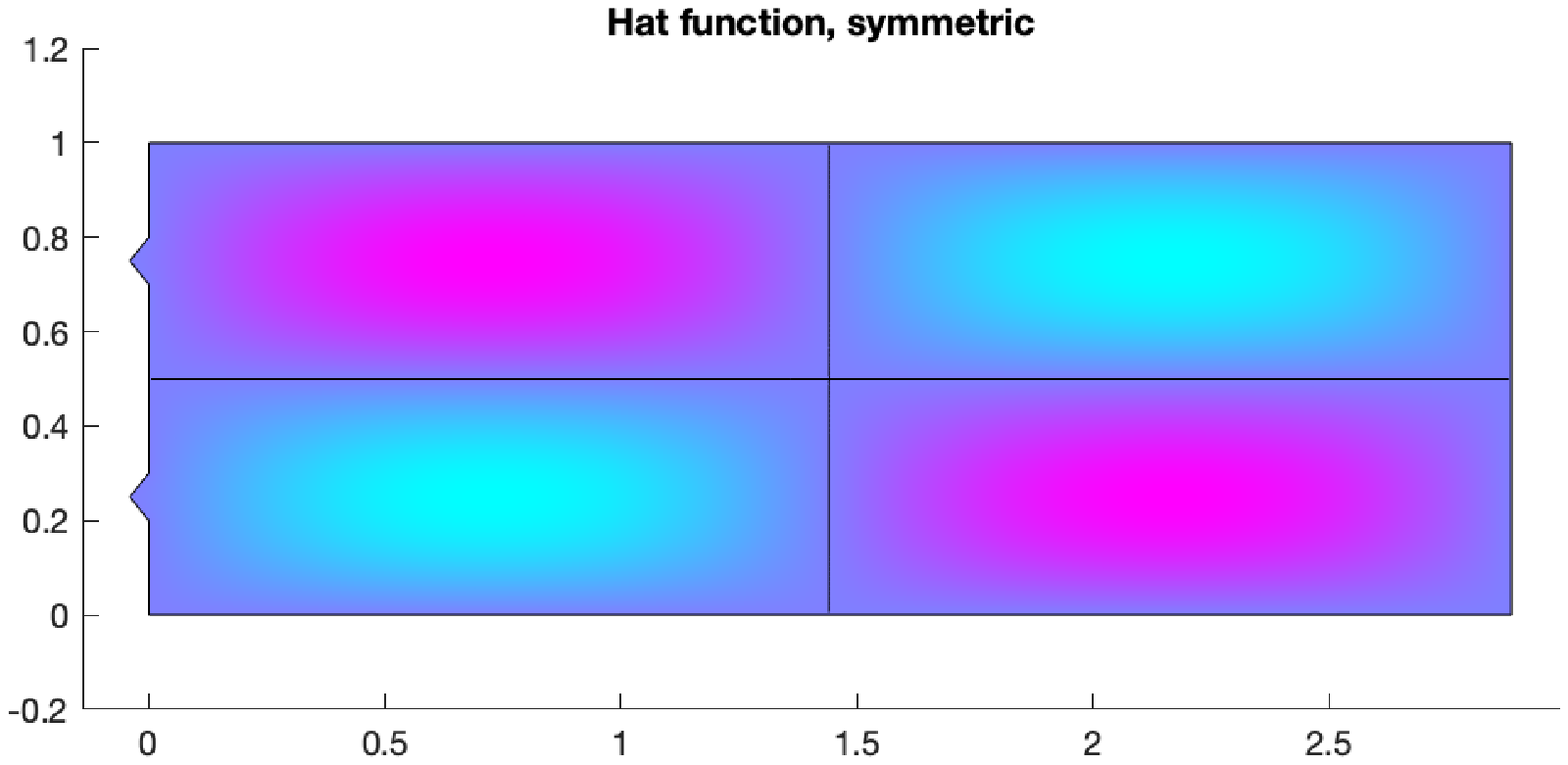}
          \caption{}
     \end{subfigure} \quad
       \begin{subfigure}{0.4\textwidth}
         \centering
          \includegraphics[width=60mm,trim={1.5cm 1.5cm 1.5cm 1.5cm},clip]{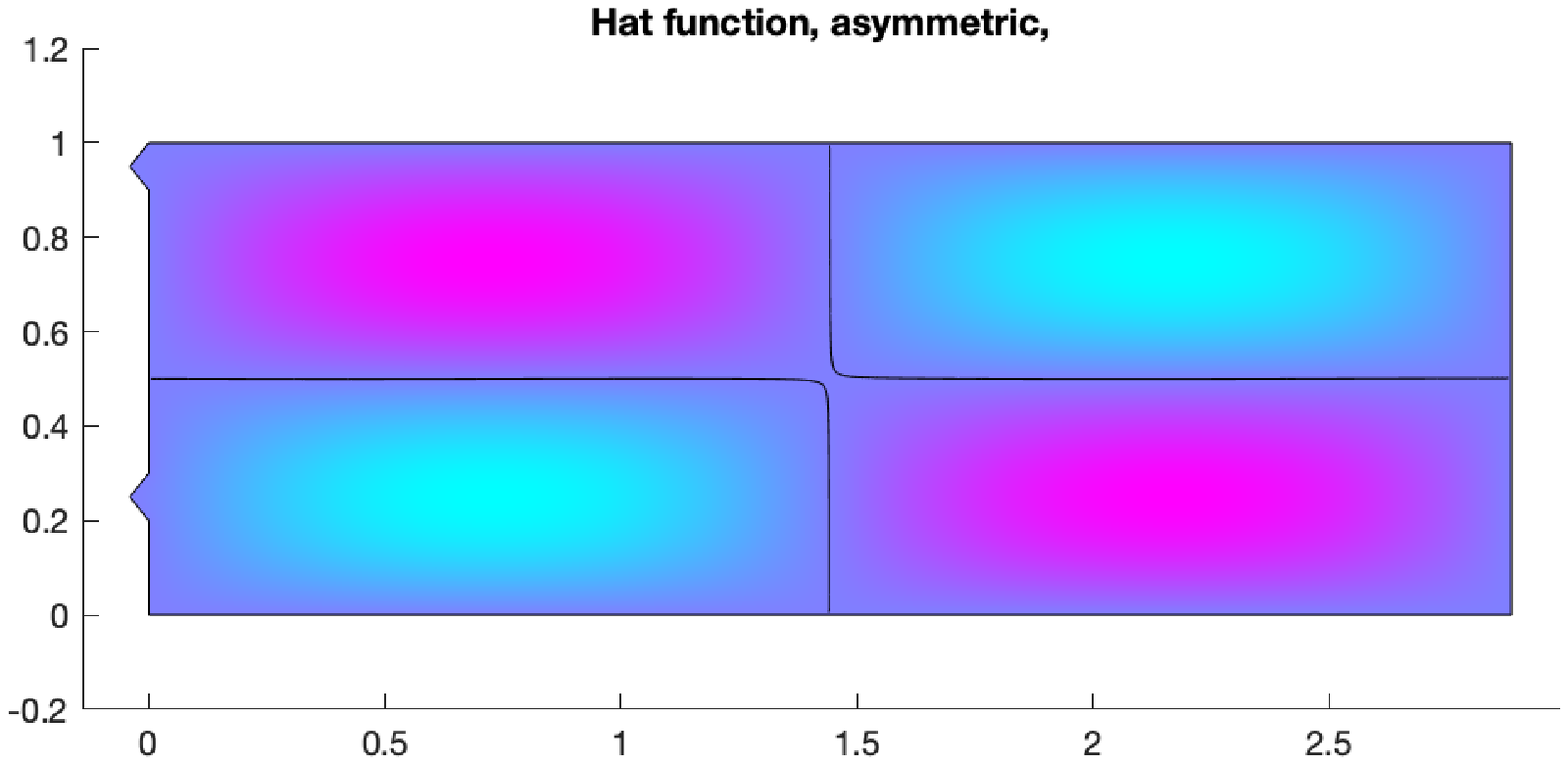} 
          \caption{}
     \end{subfigure} \\
     \begin{subfigure}{0.4\textwidth}
         \centering
          \includegraphics[width=60mm,trim={1.5cm 1.5cm 1.5cm 1.5cm},clip]{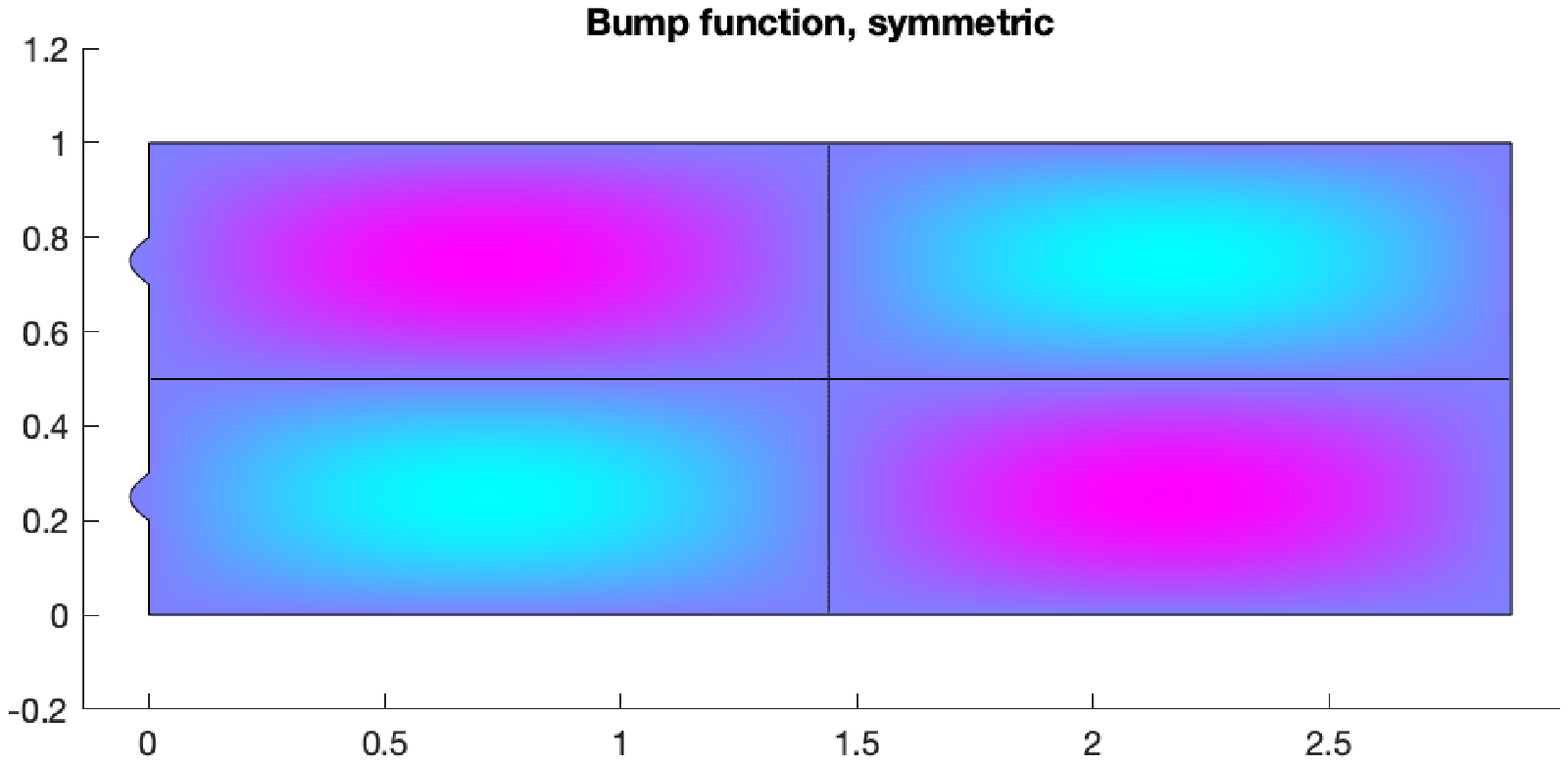}
          \caption{}
     \end{subfigure} \quad
       \begin{subfigure}{0.4\textwidth}
         \centering
          \includegraphics[width=60mm,trim={1.5cm 1.5cm 1.5cm 1.5cm},clip]{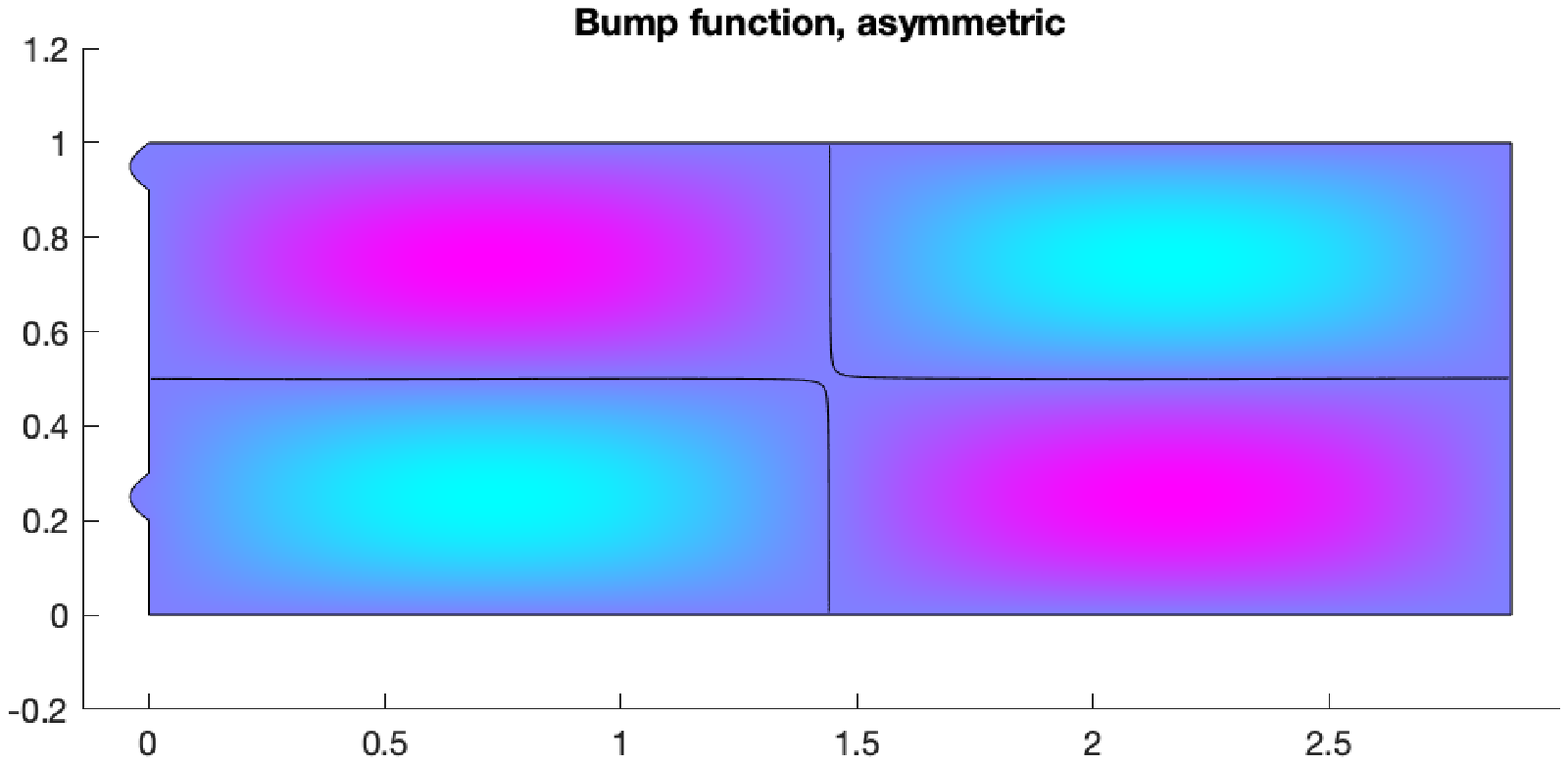} 
          \caption{}
     \end{subfigure}
\caption{Symmetric and asymmetric boundary perturbations, $N = 5/\sqrt{3}$. In (i), (ii), the left boundary is given by $0.04\phi_L(y)$, with $\phi_L(y) = \cos(6\pi y)$ in the symmetric case and $\phi_L(y) = \sin(6\pi y)$ in the asymmetric case. In (iii)-(vi), hats and bumps of radius $0.05$ and amplitude $0.04$, and are placed at $y = 0.25, 0.75$ in the symmetric boundary, and $y = 0.25, 0.95$ in the asymmetric.}
\label{fig:symmetry}
\end{figure}

\begin{figure}[ht] 
 \includegraphics[width=52mm,trim={2cm 2cm 2cm 2cm},clip]{figures/figure-1-symmetric.eps}  \!\!\!\!\!\!\!\! \includegraphics[width=40mm,trim={1.5cm 1.5cm 1.5cm 1.5cm},clip]{figures/figure-1-asymmetric.eps} \ 
  \includegraphics[width=40mm,trim={1.5cm 1.5cm 1.5cm 1.5cm},clip]{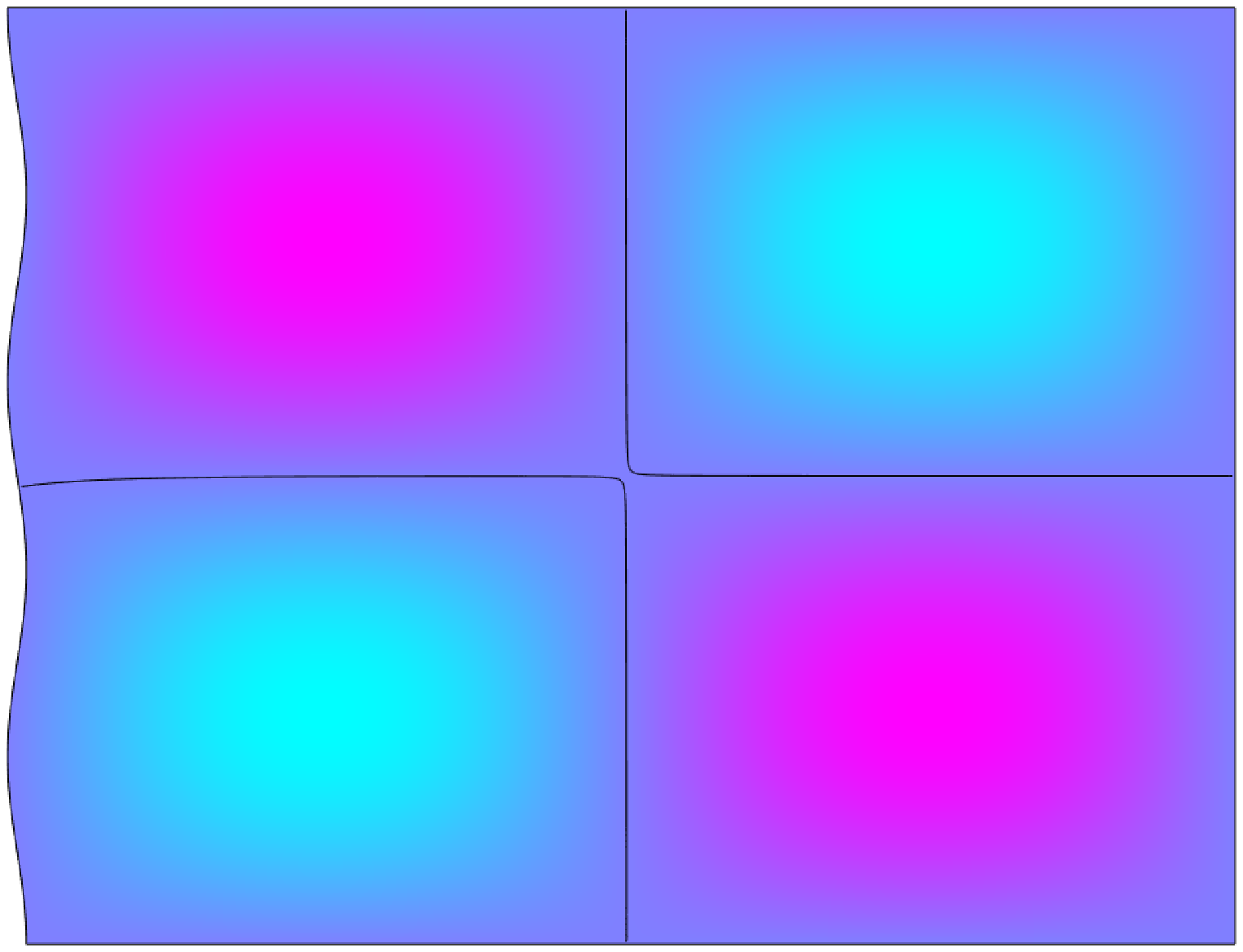}
\caption{Boundary perturbations of $\phi_L(y) = \cos(6\pi y)$ (left, symmetric), $\sin(6\pi y)$ (middle, asymmetric), and $\cos(5\pi y)$ (right, asymmetric). The boundary integral in Assumption \ref{assum:phi} vanishes for $\phi_L(y) = \cos(5\pi y),$ and hence the nodal set opening is much smaller.}
\label{fig:non-zero-bdry-integral}
\end{figure}

\subsection{Role of Assumption \ref{assum:simple} and Change in Orientation about Resonant $N$}
We recall that by Assumption \ref{assum:simple}, the aspect ratio $N$ is chosen to avoid a degenerate eigenvalue. Further, according to Proposition \ref{prop:hyperbola}, we see that the sign of $v_2'(\tfrac{N}{2})/v_1(\tfrac{N}{2})$ 
determines the orientation of the nodal set. In this section we investigate the behavior of the nodal set when Assumption \ref{assum:simple} does not hold and confirm the observation about the orientation of the nodal set.

To investigate this, we set $\eta = \tfrac{1}{32}$, $\phi_L(y) = \sin(6\pi y)$, and consider aspect ratios near the resonant value of close to $N =  \sqrt{15} \approx 3.8730$. For such $N$,  we have that $\min_{k}\left|3 + \tfrac{4}{N^2} - \tfrac{k^2}{N^2} \right| =0,$ where the minimum is attained for $k = 7.$ Near this aspect ratio, the nodal set deforms greatly (see Figure \ref{fig:resonant-N}).

The plots in Figure \ref{fig:resonant-N} also show that the nodal set switches orientation as the aspect ratio passes through a resonant value, where the axis along the opening changes from $[1, -1]^T$ to $[1,1]$. This corresponds to a sign change of $v_1(\tfrac{N}{2}) = v_1(0)/\left(2\cos\left(\mu_1 \tfrac{N}{2}\right)\right)$, while $v_2'(\tfrac{N}{2})$ remains negative.

\begin{figure} 
\begin{tabular}{cc}
\includegraphics[width=60mm,trim={2.0cm 2.5cm 2.5cm 2.5cm},clip]{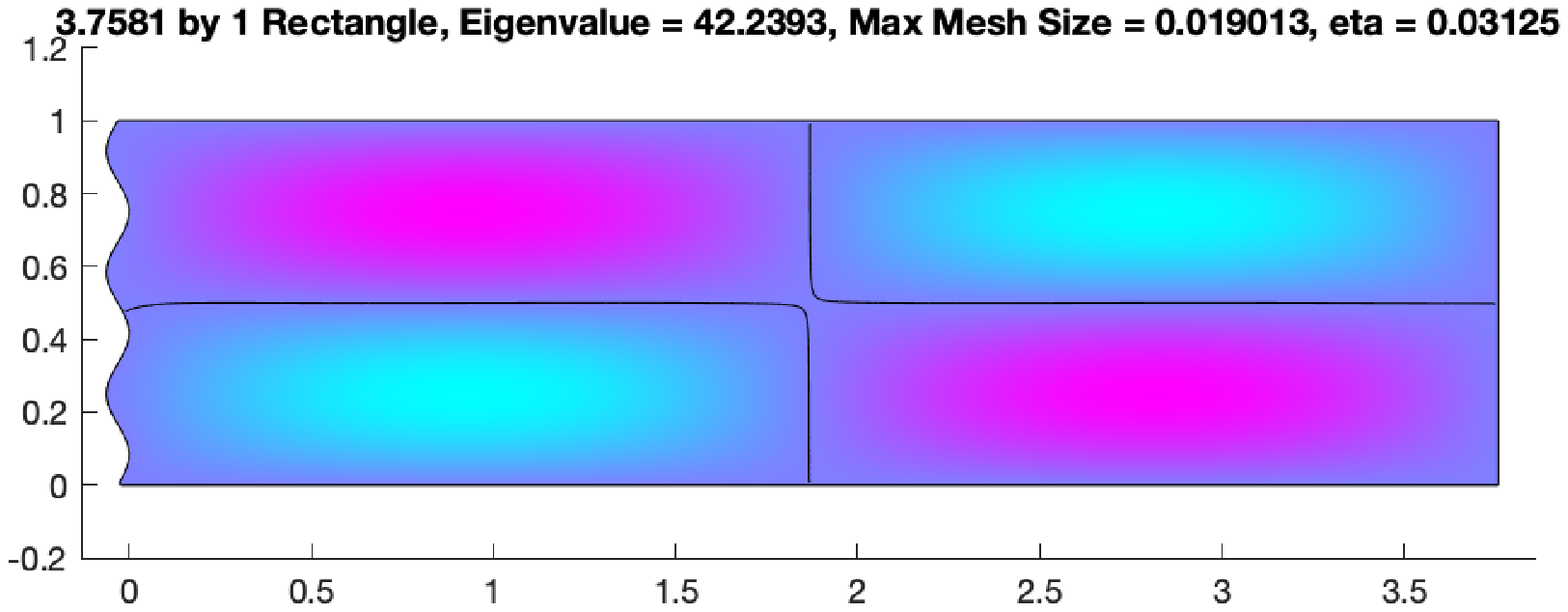} &
\includegraphics[width=60mm,trim={2.0cm 2.5cm 2.5cm 2.5cm},clip]{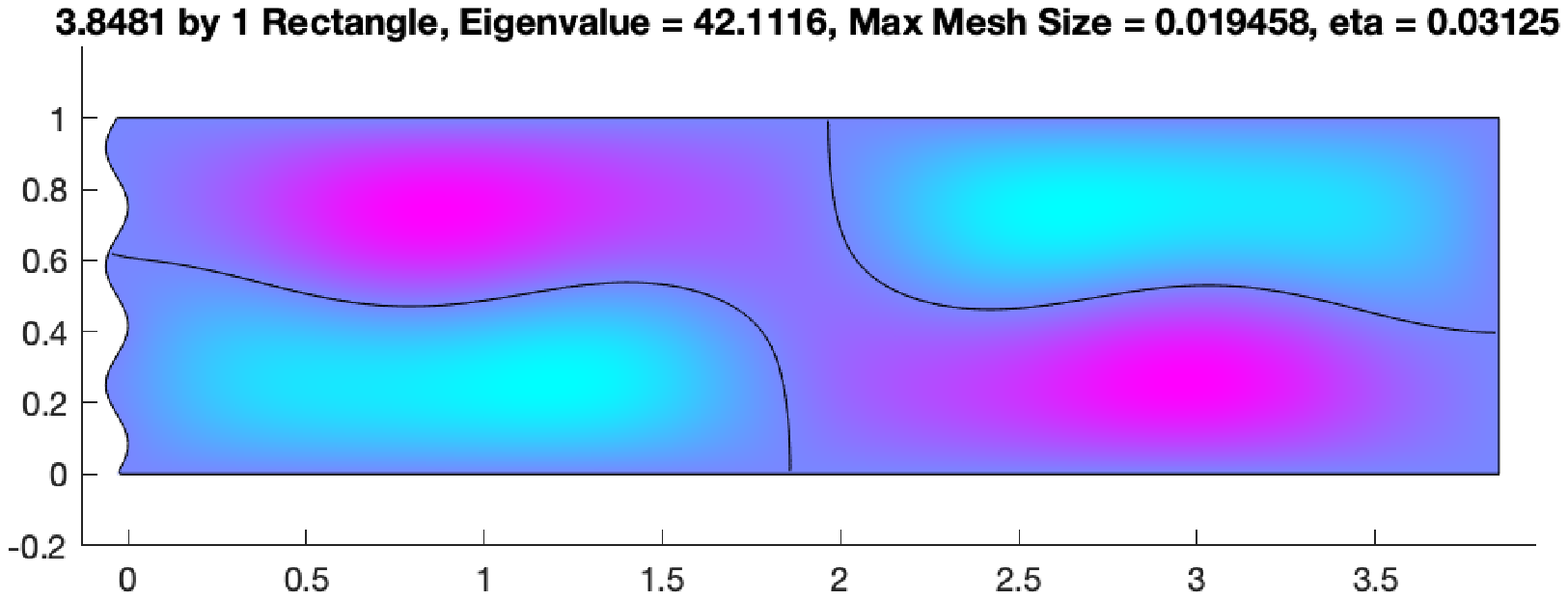}\\
\includegraphics[width=60mm,trim={2.0cm 2.5cm 2.5cm 2.5cm},clip]{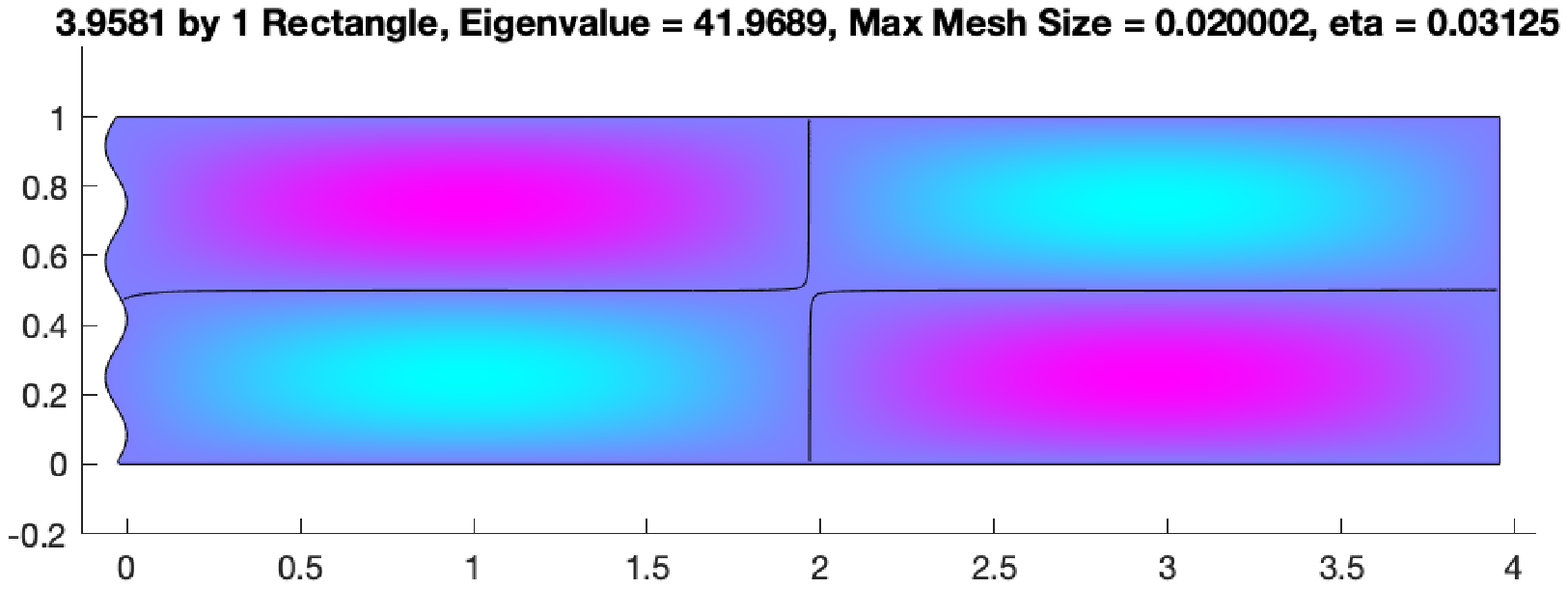}
     & \includegraphics[width=60mm,trim={2.0cm 2.5cm 2.5cm 2.5cm},clip]{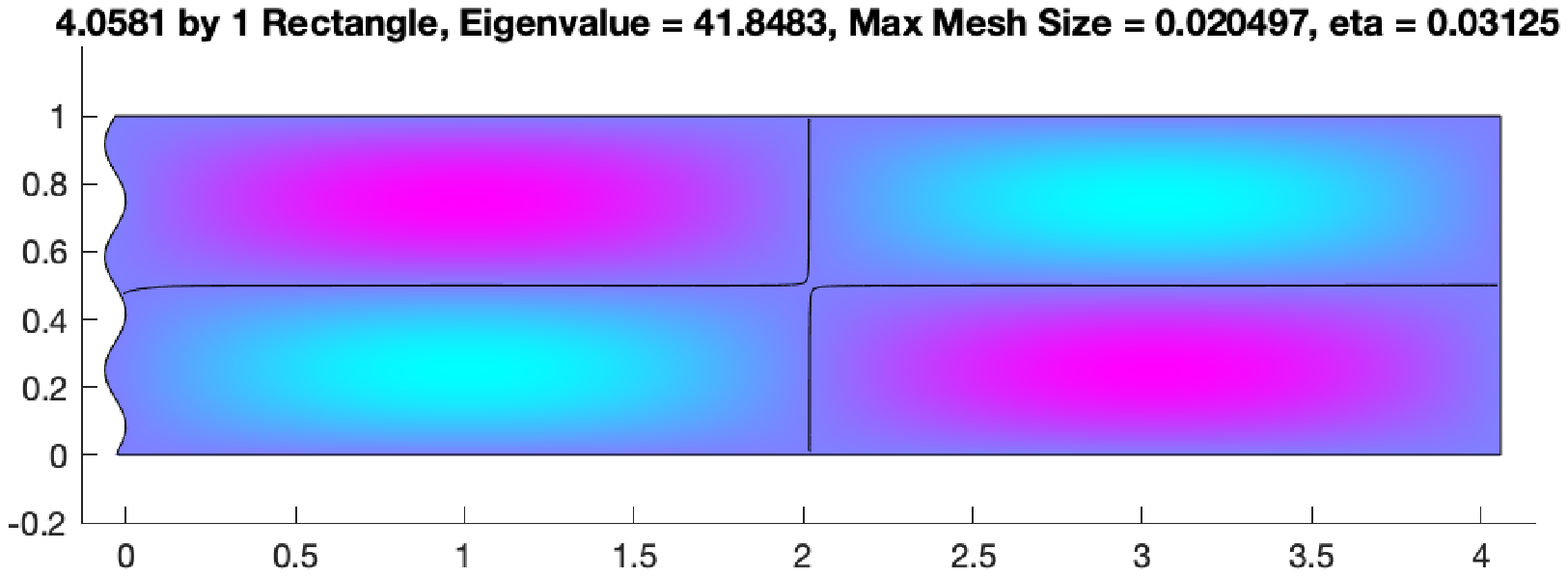}
\end{tabular}
\caption{Plots for $N$ near the resonant aspect ratio $N=\sqrt{15}$.}
\label{fig:resonant-N}
\end{figure}

\subsection{Nodal Set Gap Independence of $N$}
As shown in Theorem \ref{thm:nodal-gap-estimate}, the separation of the two branches of the nodal set is comparable to $\sqrt{\eta}$, independent of $N$. To verify this, we fix $\eta = \tfrac{1}{2}$, $\phi_L(y) = \sin(6\pi y)$, and plot the desired mode for various aspect ratios to see if the nodal gap remains comparable. Since the nodal set opening depends on the relative position of $N$ to a degenerate eigenvalue, we control for the resonance effect by choosing $N$ such that $\left|\sin\left(\mu_1 N\right)\right|$ is close to $1.$ This ensures that the expression for $v_1(x)$ in \eqref{eqn:v1} is well defined. Gradient descent/ascent is used to choose such $N$. The mesh granularity is tuned for each domain to ensure that the maximum mesh element sizes are comparable across all tested aspect ratios. Then, for each $N$ tested, the two sets of nodal points for each branch near the center are approximately identified by finding all mesh points $(x,y)$ near the extremal point of each branch, satisfying $|v(x,y)| < \varepsilon,$ where $\varepsilon \approx 10^{-7}$ varies slightly for each $N$. Given these two sets of nodal points on each branch, labeled $P, Q$, the opening distance is approximately calculated (up to numeric error) as $d := d(P,Q) = \min_{(p,q) \in P \times Q} d(p,q)$, where $d(\cdot,\cdot)$ is the Euclidean metric. The extremal points on each branch are identified up to small error, and distances $d$ are reported below (see Figure \ref{fig:opening-independence}, end of document, and Table \ref{table:separation}).

The distances shown in Table \ref{table:separation} are comparable in order of magnitude across the various length-scales. This agrees with the opening of the nodal set described in Theorem \ref{thm:nodal-gap-estimate}, and note that our numerics are for moderately sized $N$. This is consistent with the fact that in several of our proofs, we did not need to take $N$ sufficiently large to obtain our needed results and estimates. More broadly, this suggests that for many perturbations our results will still hold for moderately sized $N$, as discussed in Remark \ref{Nrem}.

\begin{figure}[ht] 
\begin{tabular}{cc} 
  \includegraphics[width=60mm,trim={2.0cm 2.5cm 2.0cm 2.5cm},clip]{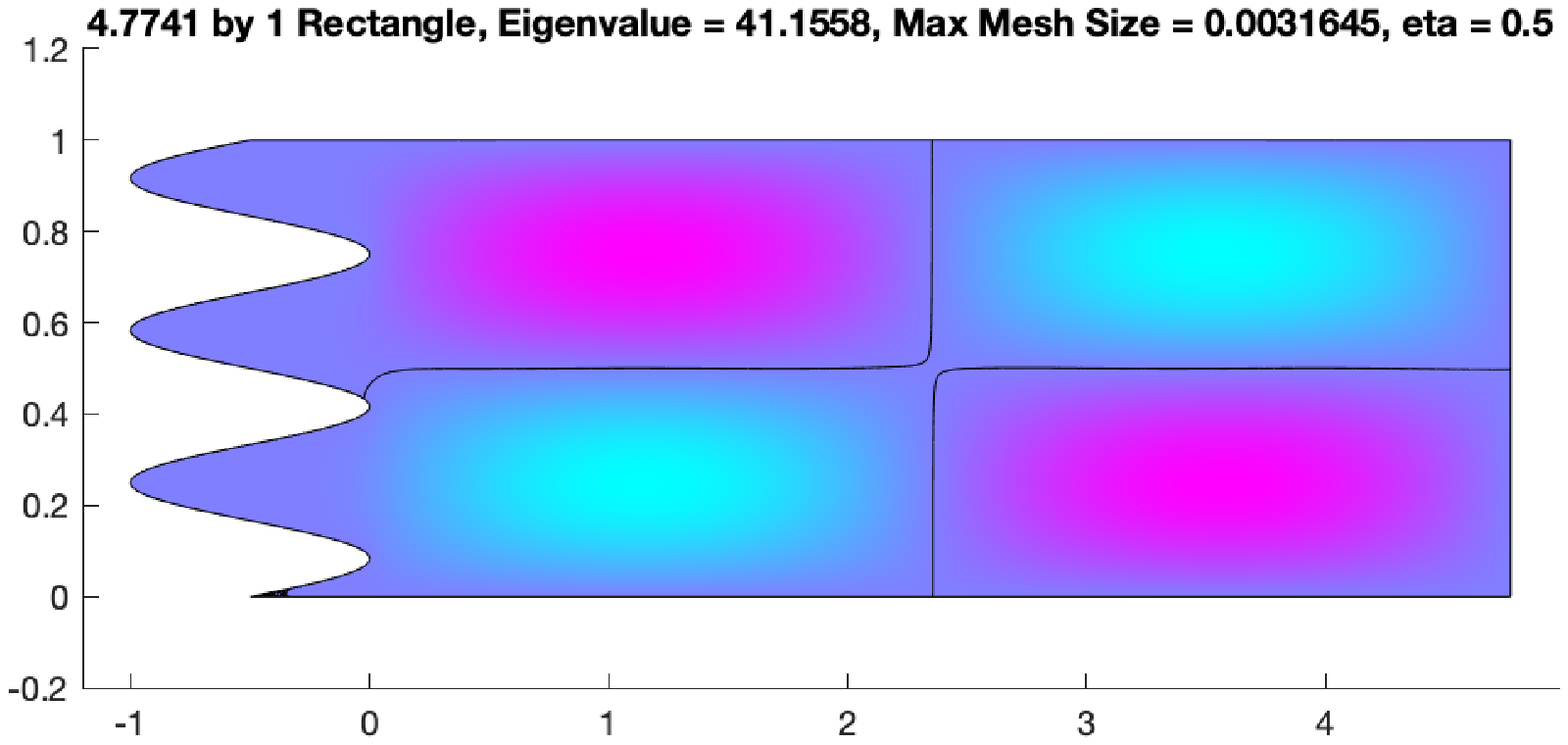} & \includegraphics[width=60mm,trim={2.0cm 2.5cm 2.0cm 2.5cm},clip]{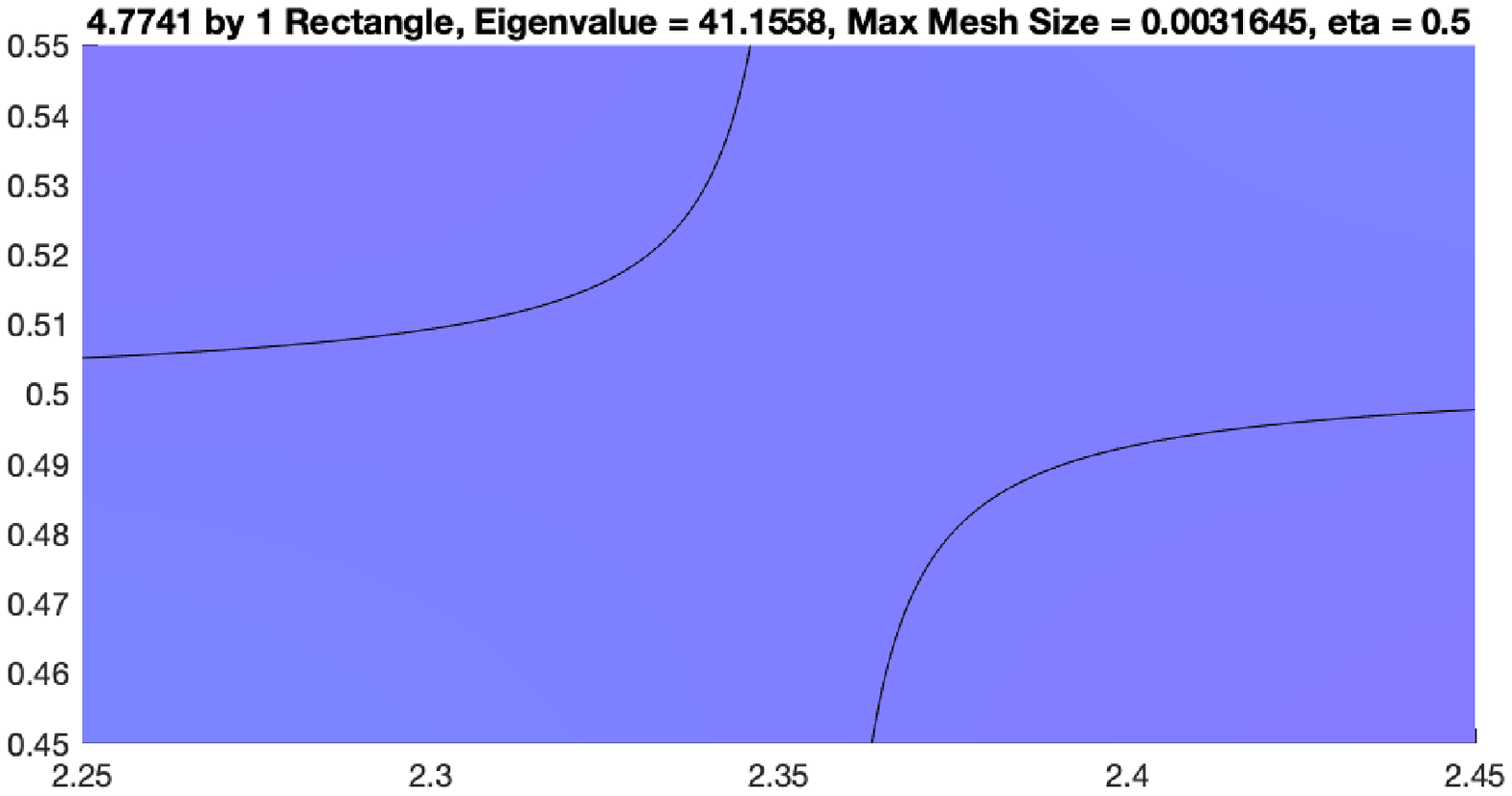} \\
(i.a) & (i.b) \\[6pt]
 \includegraphics[width=60mm,trim={2.0cm 2.5cm 2.0cm 2.5cm},clip]{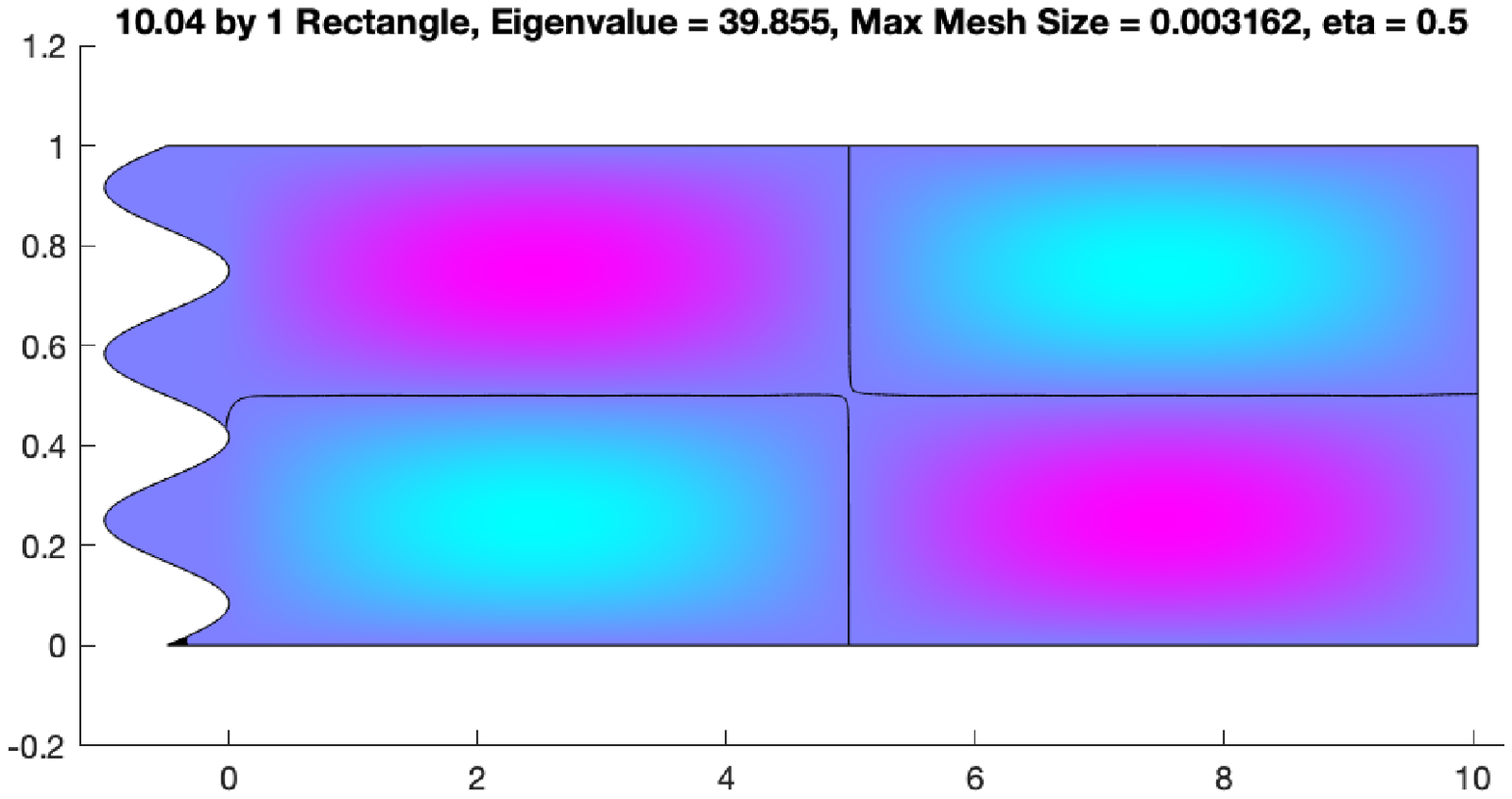} &  \includegraphics[width=60mm,trim={2.0cm 2.5cm 2.0cm 2.5cm},clip]{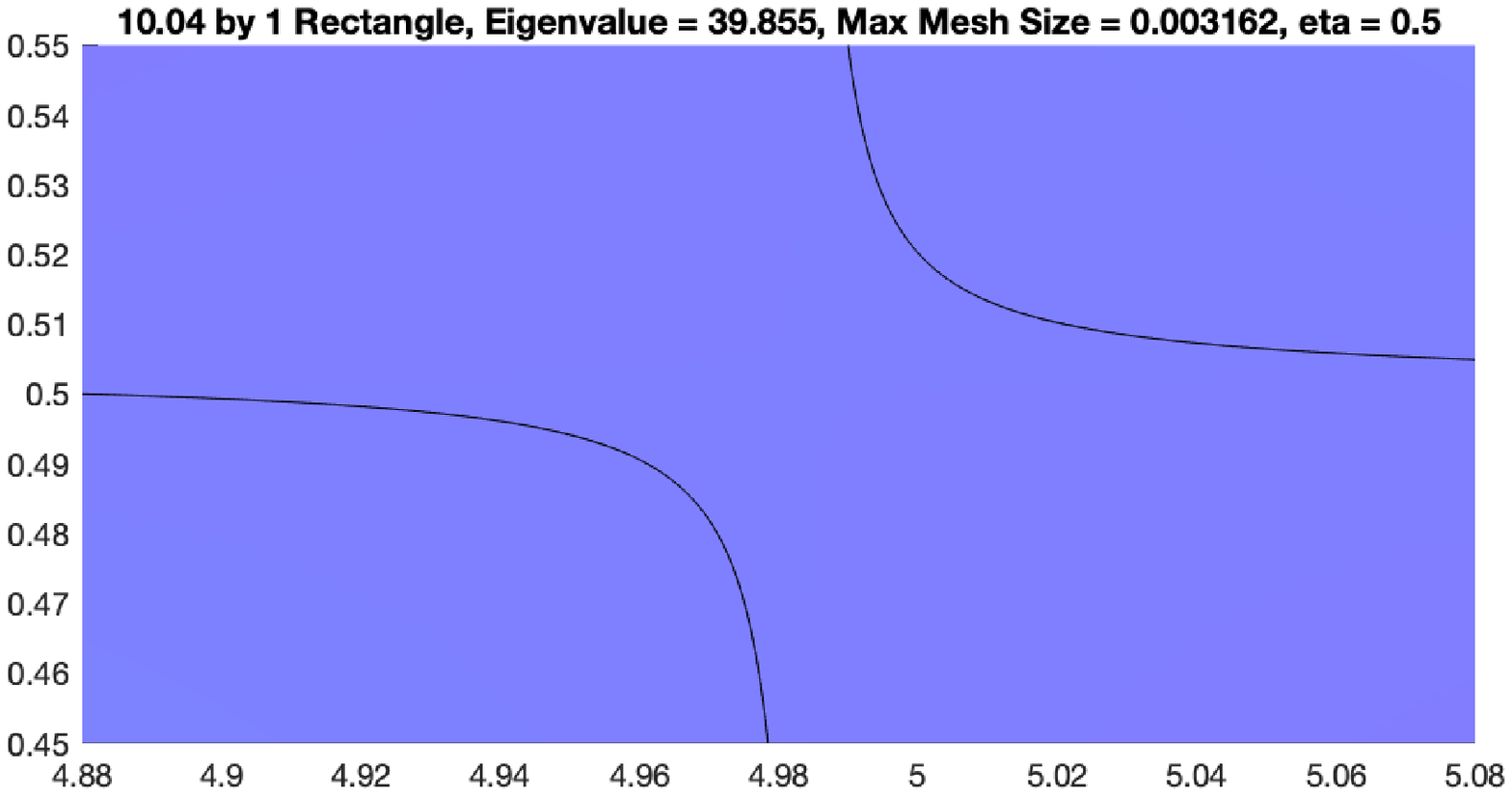}\\
(ii.a) & (ii.b)  \\[6pt]
 \includegraphics[width=60mm,trim={2.0cm 2.5cm 2.0cm 2.5cm},clip]{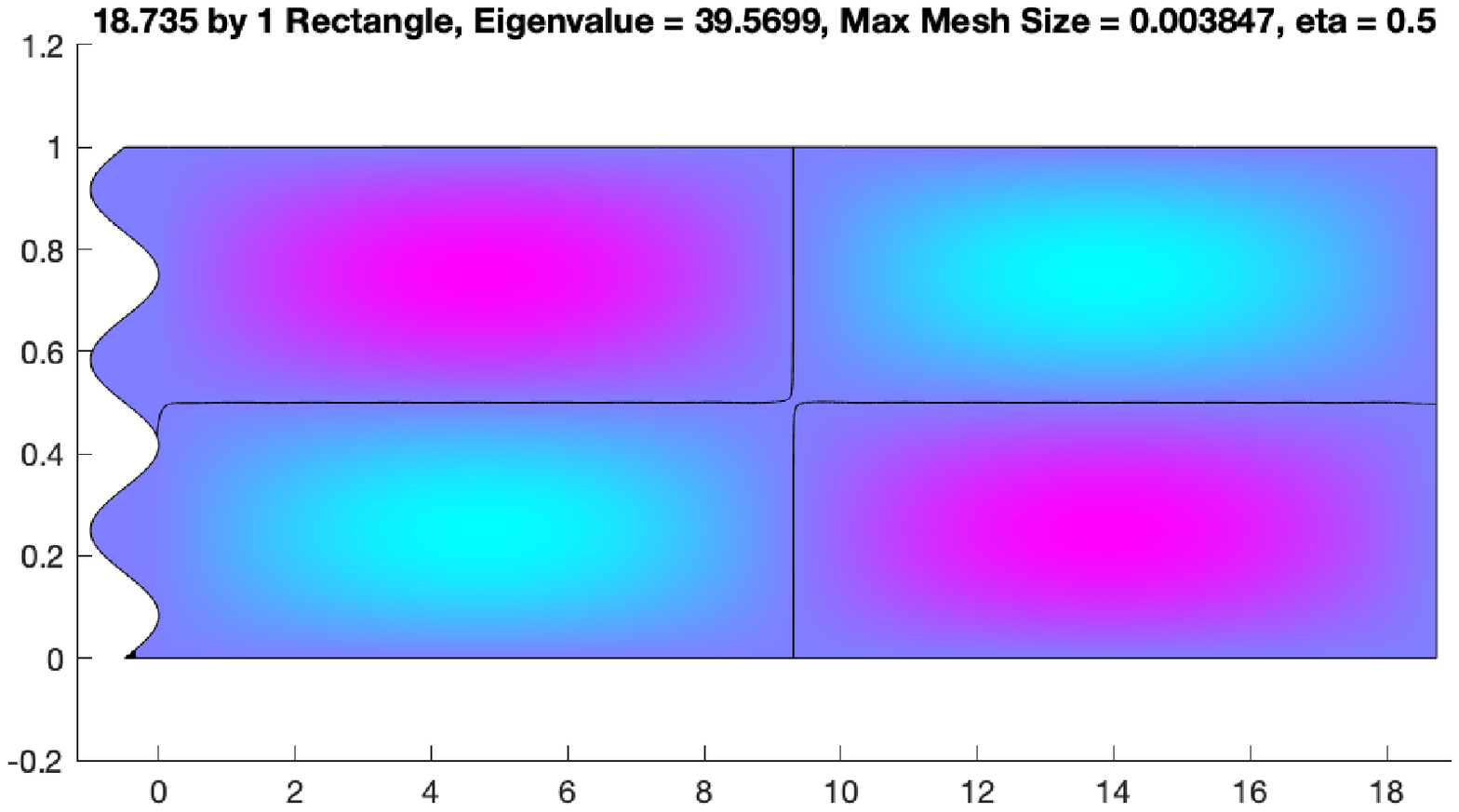} &  \includegraphics[width=60mm,trim={2.0cm 2.5cm 2.0cm 2.5cm},clip]{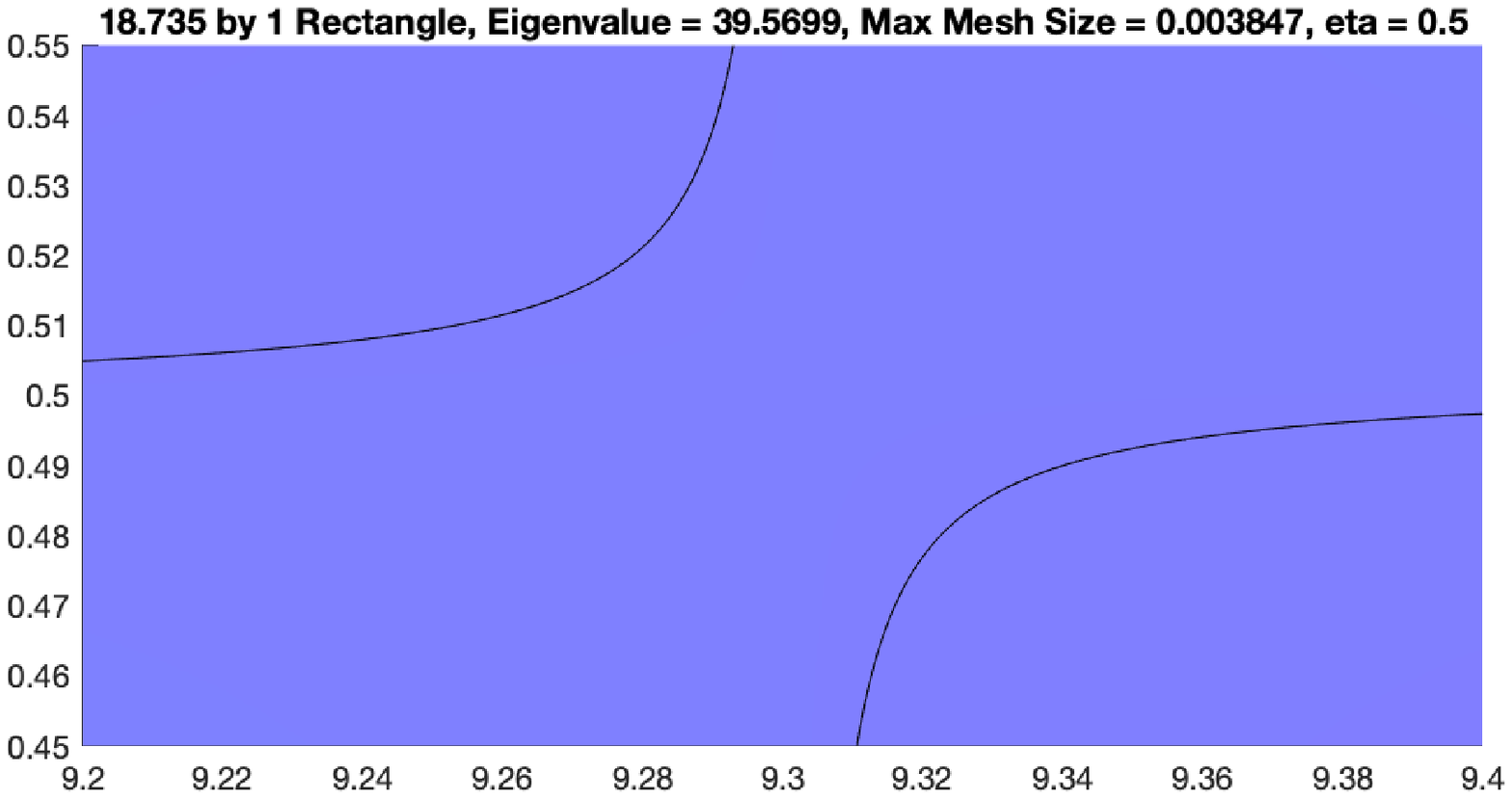}\\
(iii.a) & (iii.b) \\[6pt]
  \includegraphics[width=60mm,trim={2.0cm 2.5cm 2.0cm 2.5cm},clip]{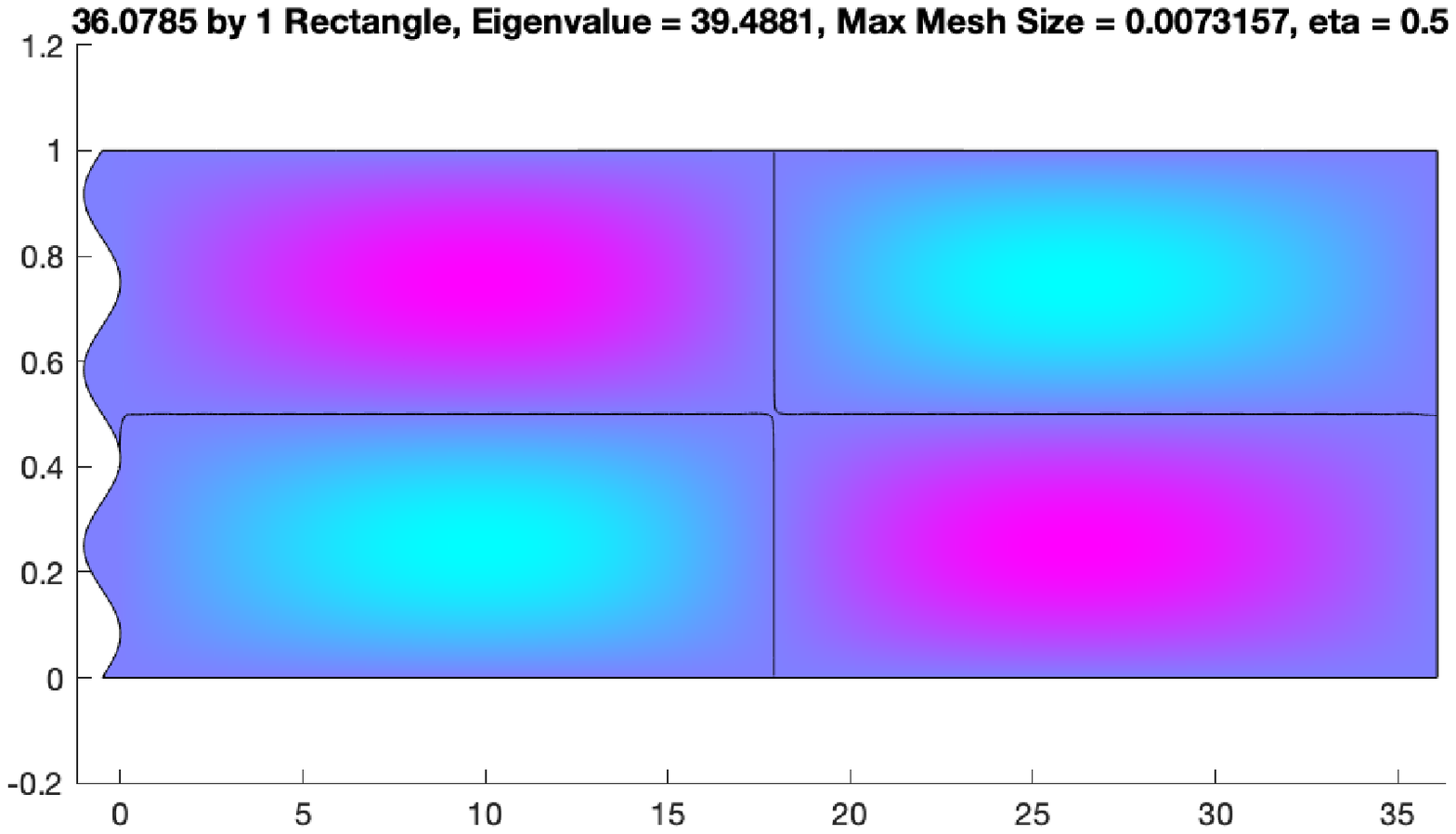} &  \includegraphics[width=60mm,trim={2.0cm 2.5cm 2.0cm 2.5cm},clip]{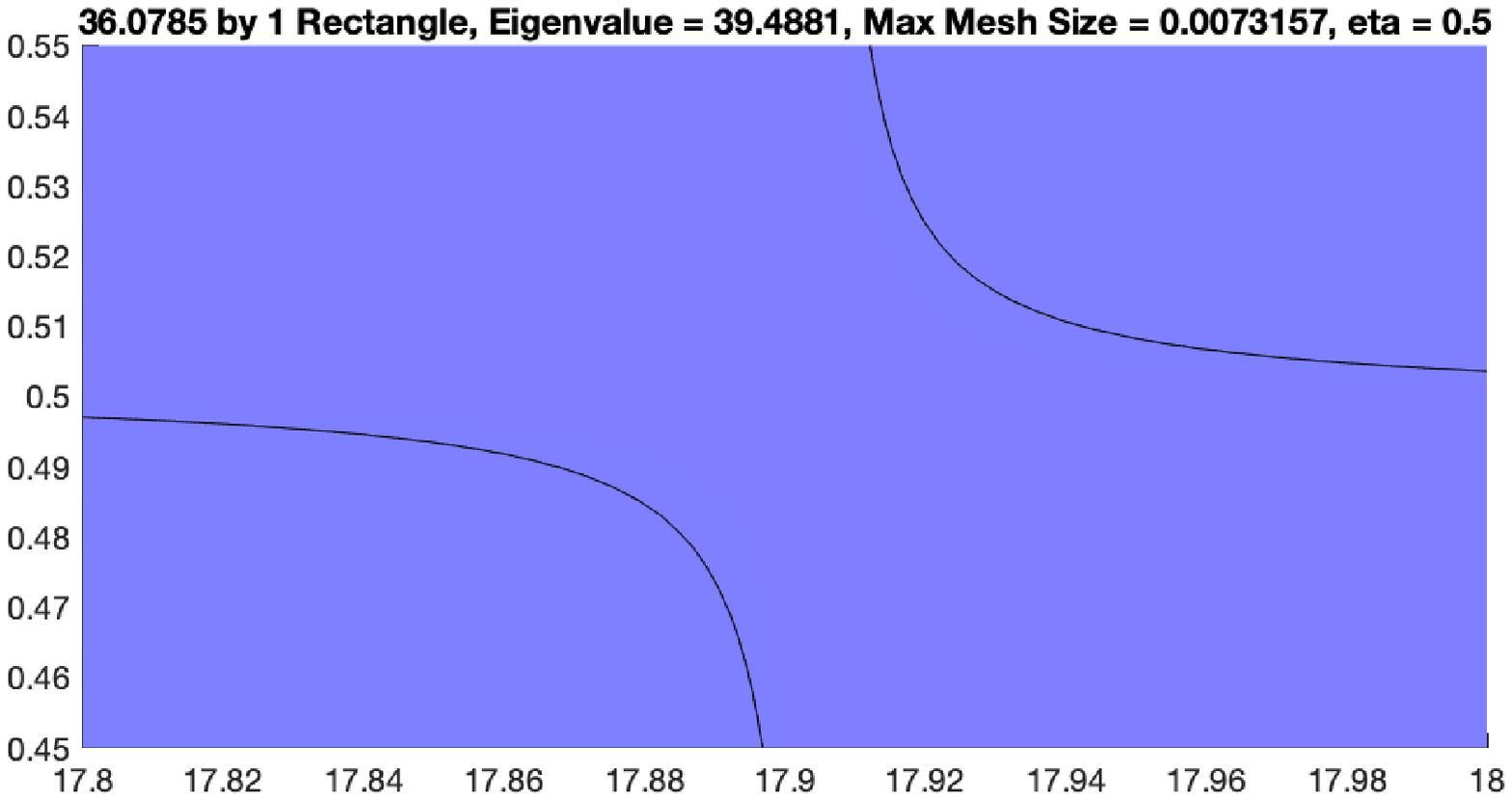}\\
(iv.a) & (iv.b)  \\[6pt]
\end{tabular}
\caption{Plots for varying aspect ratios with fixed boundary perturbation, as well as detail of the nodal opening.}
\label{fig:opening-independence}
\end{figure}

\begin{table}[ht] 
\begin{tabular}{c|cccc}
$N$ &  $d$  & Max. Mesh Size ($\times 10^{-3}$) & $\mu$  &  $\left|\sin(\mu_1 N) \right|$\\
\hline
 4.7741     &  0.0552    & 3.1645    & 41.1558     & 1.0000 \\
 10.0400    &  0.0447    & 3.162     & 39.8550      & 1.0000 \\
 18.7350    &  0.0545     & 3.847    & 39.5699     & 1.0000 \\
 36.0785    &  0.0480     & 7.3157    & 39.4881      & 1.0000
\end{tabular}
\caption{Computed separation distances $d$, with mesh sizes, eigenvalues $\mu$, and the size of resonance effect for each $N$. Recall $\mu_1^2 = \mu - \pi^2.$} \label{table:separation}
\end{table}

\newpage \,
\newpage \,

\bibliographystyle{plain}
\bibliography{nodal}

\end{document}